\documentclass[11pt,reqno]{amsart}
\usepackage{fullpage}

\usepackage{amssymb,amsmath,amsthm, amsfonts}
\usepackage{comment}
\usepackage{mathrsfs}
\usepackage{caption}
\usepackage{subcaption}
\usepackage{float}
\usepackage[dvipsnames]{xcolor}
\usepackage{amsmath}
\usepackage{graphicx}
\usepackage{mathtools}

\usepackage[all,knot]{xy}
\usepackage{tabularx}
\usepackage{setspace}
\usepackage{todonotes}
\usepackage{tikz}
\usepackage{mathtools}
\usepackage{multicol}

\newtheorem{theorem}{Theorem}

\newtheorem{lemma}[theorem]{Lemma}
\theoremstyle{definition}
\newtheorem{example}[theorem]{Example}

\newtheorem{remark}[theorem]{Remark}

\newtheorem{proposition}[theorem]{Proposition}

\theoremstyle{definition}
\numberwithin{equation}{section}

\newcommand{\Z}{\mathbb{Z}}
\newcommand{\N}{\mathbb{N}}
\newcommand{\C}{\mathbb{C}}
\newcommand{\Arg}{\mathrm{Arg}}
\newcommand{\Log}{\mathrm{Log}}

\makeatletter
\@namedef{subjclassname@2020}{%
	\textup{2020} Mathematics Subject Classification}
\makeatother

\title{Asymptotics for the twisted eta-product and applications to sign changes in partitions}
\author{Walter Bridges, Johann Franke, and Taylor Garnowski}
\address{University of Cologne, Department of Mathematics and Computer Science, Weyertal 86-90, 50931 Cologne, Germany}
\email{wbridges@uni-koeln.de}
\email{{jfrank12@uni-koeln.de}}
\email{{tgarnows@uni-koeln.de}}

\subjclass[2020]{11P82, 11P83}

\keywords{circle method, partitions, asymptotics, sign-changes, secondary terms}

\begin{document}
\maketitle
\begin{abstract}
    We prove asymptotic formulas for the complex coefficients of $(\zeta q;q)_\infty^{-1}$, where $\zeta$ is a root of unity, and apply our results to determine secondary terms in the asymptotics for $p(a,b,n)$, the number of integer partitions of $n$ with number of parts congruent $a$ modulo $b$.  Our results imply that, as $n \to \infty$, the difference $p(a_1,b,n)-p(a_2,b,n)$ for $a_1 \neq a_2$ oscillates like a cosine when renormalized by elementary functions.  Moreover, we give asymptotic formulas for arbitrary linear combinations of $\{p(a,b,n)\}_{1 \leq a \leq b}.$
\end{abstract}
\section{Introduction and statement of results}

Let $n$ be a positive integer.  A \textit{partition} of $n$ is a weakly decreasing sequence of positive integers that sum to $n$. The number of partitions of $n$ is denoted by $p(n)$. For example, one has $p(4) = 5$, since the relevant partitions are 
\begin{align*}
    4, \qquad 3+1, \qquad 2+2, \qquad 2+1+1, \qquad 1+1+1+1. 
\end{align*}
When $\lambda_1 + \cdots + \lambda_r = n$, the $\lambda_j$ are called the \textit{parts} of the partition $\lambda = (\lambda_1, ..., \lambda_r)$, and we write $\lambda \vdash n$ to denote that $\lambda$ is a partition of $n$. The partition function has no elementary closed formula, nor does it satisfy any finite order recurrence. However, when defining $p(0) := 1$, its generating function has the following product expansion
\begin{align}
\label{partitionproduct} \sum_{n \geq 0} p(n)q^n = \prod_{k \geq 1} \frac{1}{1-q^k} = \frac{q^{\frac{1}{24}}}{\eta(\tau)},
\end{align}
where as usual $q := e^{2\pi i \tau}$ and $\eta(\tau)$ denotes the Dedekind eta function. In \cite{hardyramanujan}, Hardy and Ramanujan used \eqref{partitionproduct} to show the asymptotic formula 
\begin{align}
\label{partitionasy} p(n) \sim \frac{1}{4\sqrt{3} n} \exp\left( \pi \sqrt{\frac{2n}{3}}\right), \qquad n \rightarrow \infty.
\end{align}

For their proof they introduced the circle method, certainly one of the most useful tools in modern analytic number theory. The now so-called Hardy--Ramanujan circle method uses modular type transformations to obtain a divergent asymptotic expansion whose truncations approximate $p(n)$ up to small errors. A later refinement by Rademacher \cite{Rademacher}, exploiting the modularity of $\eta(\tau)$, provides a convergent series for $p(n)$.

In this work we study the distribution of ordinary partitions whose number of parts are congruent to some residue class $a$ modulo $b$.\footnote{A similar sounding, but entirely different problem is to study the total number of parts that are all congruent to $a$ modulo $b$.  This problem was studied in detail by Beckwith--Mertens \cite{BeckwithMertens2,BeckwithMertens1} for ordinary partitions and recently by Craig \cite{Craig} for distinct parts partitions.} The number of such partitions $\lambda \vdash n$ is denoted by $p(a,b,n)$.\footnote{Note that the number of partitions $\lambda \vdash n$ with largest part $\ell(\lambda)$ equals the number of partitions with number of parts $\ell(\lambda)$ (see also \cite{andrewsbook}, Ch. 1), so the results of this paper can be reformulated taking number of parts into account instead.} It is well-known 
that the numbers $p(a,b,n)$ are asymptotically equidistributed; i.e., 
\begin{align} \label{pequa}
    p(a,b,n) \sim \frac{p(n)}{b}, \qquad n \to \infty. 
\end{align}
The proof of \eqref{pequa} begins by writing the generating function for $p(a,b,n)$ in terms of non-modular eta-products twisted with roots of unity modulo $b$ which is a direct consequence of the orthogonality of roots of unity:
\begin{align} \label{pgen}
    \sum_{n \geq 0} p(a,b,n)q^n = \frac{1}{b} \left( \frac{q^{\frac{1}{24}}}{\eta(\tau)} + \sum_{b-1 \geq j \geq 1}\frac{\zeta_b^{-ja}}{ \left(\zeta_b^jq; q\right)_\infty}\right),
\end{align}
where $(a;q)_\infty := \prod_{n \geq 0} (1-aq^n)$ denotes the usual $q$-Pochhammer symbol and $\zeta_b := e^{\frac{2\pi i}{b}}$. Since the first term does not depend on $a$ and dominates the other summands, \eqref{pequa} can be seen as a corollary of \eqref{pgen}. Similar results have also been proved for related statistics. For example, the rank of a partition is defined to be the largest part minus the number of parts; Males \cite{Males} showed that the Dyson rank function $N(a,b,n)$ (the number of partitions with rank congruent to $a$ modulo $b$) is asymptotically equidistributed. Males' proof uses Ingham's Tauberian theorem \cite{Ingham} and exploits the modularity of the generating function of $N(a,b,n)$, which is given in terms of twisted generalizations of one of Ramanujan's famous mock theta functions. 
In contrast, the twisted eta-products in \eqref{pgen} lack modularity.

If we now consider differences, $p(a_1,b,n)-p(a_2,b,n)$, the main terms in \eqref{pgen} cancel and the behavior must be determined by secondary terms. In the following example we look at the differences of two modulo 5 partition functions. 

\begin{example}
Consider the case $a_1 = 1$, $a_2 = 4$ and $b = 5$. The $q$-series of the differences has the following oscillating shape 
\begin{align*}
    \sum_{n \geq 0} \left(p(1,5,n) - p(4,5,n)\right)q^n  = \ &  q+q^2+q^3 -q^7-2 q^8 -3 q^9-4 q^{10} -4 q^{11}-5 q^{12} -6 q^{13}-7 q^{14} \\ 
    &-7 q^{15}-7 q^{16}  - \cdots + 2q^{22} + \cdots + 109q^{40} + \cdots + 11q^{48} - 24q^{49} \\ 
    \vspace{1cm}
    &- \cdots - 1998q^{75} - \cdots - 266q^{85} + 163q^{86} + \cdots + 40511 q^{120} + \cdots \\ 
    \vspace{1cm}
    &+ 3701 q^{133} -3587 q^{134} - \cdots .
\end{align*} 

\begin{figure}[H]
         \centering
         \includegraphics[width=120mm,height= 55mm]{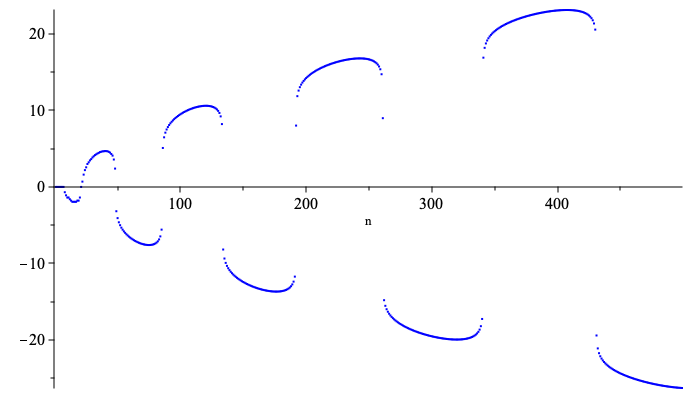}
        \caption{For $b= 5$, $a_1 = 1$, and $a_2=4$, we plot  the difference $p(1,5,n)-p(4,5,n)$ with log-scaling. The abrupt vertical changes in the graph indicate the location of the sign changes in the sequence. }
\label{fig:log145}
\end{figure}
\end{example}

Our first result predicts this oscillation as follows.  As usual, we define the dilogarithm for $|z|\leq 1$ by the generating series
$$
\mathrm{Li}_2(z):=\sum_{n \geq 1} \frac{z^n}{n^2}.
$$
Throughout we use the principle branch of the complex square root and logarithm.
\begin{theorem} \label{ThmA} Let $b \geq 5$ be an integer. Then for any two residue classes $a_1 \not\equiv a_2 \pmod{b}$, we have
\begin{align*}
    \frac{p(a_1,b,n) - p(a_2,b,n)}{B n^{-\frac34} \exp\left(2\lambda_1 \sqrt{n}\right)} = \cos\left(\beta + 2\lambda_2 \sqrt{n}\right) + o(1),
\end{align*}
as $n \to \infty$, where\footnote{Here and throughout we take the principal branch of the square-root and upcoming logarithms.} $\lambda_1 + i\lambda_2 := \sqrt{\mathrm{Li}_2(\zeta_b)}$ and $B > 0$ and $0 \leq \beta < 2\pi$ are implicitly defined by 
\begin{align*}
    Be^{i\beta} := \frac{1}{b}(\zeta_b^{-a_1} - \zeta_b^{-a_2}) \sqrt{\frac{(1 - \zeta_b)(\lambda_1 + i\lambda_2)}{\pi}}.
\end{align*}
\end{theorem}
A more general version of Theorem \ref{ThmA} holds for values $b \geq 2$, see Theorem \ref{T:simpledifferences}, where the special cases $b \in \{2,3,4\}$ have to be treated slightly differently.  Figure 2 shows that this prediction is surprisingly accurate even for small $n$.
\begin{figure}[H]
         \centering
         \includegraphics[width=120mm,height= 70mm]{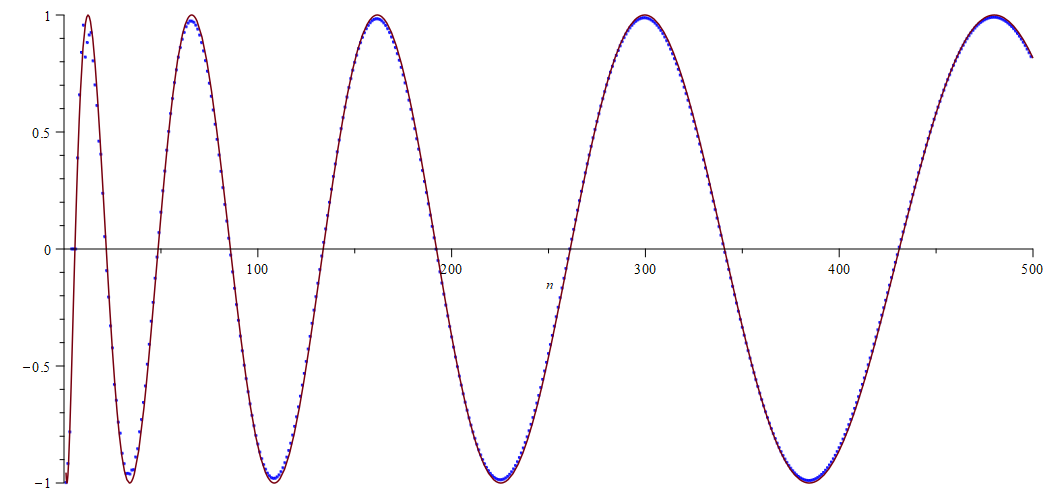}
\label{fig:difffigintro}

\caption{The plot shows the sign changes of $p(1,5,n)-p(4,5,n)$. The blue dots depict the function $\frac{p(1,5,n)-p(4,5,n)}{Bn^{-\frac34}e^{2\lambda_1\sqrt{n}}}$ and the red line is the asymptotic prediction  $\cos\left(\beta + 2\lambda_2 \sqrt{n}\right)$. The approximate values of the constants are $B \approx 0.23268$, $\beta \approx 1.4758$, $\lambda_1 \approx 0.72984$, and $\lambda_2\approx 0.68327$.}
\end{figure}

 The proof makes use of \eqref{pgen} and a detailed study of the coefficients $\sum_{n \geq 0} Q_n(\zeta)q^n := (\zeta q; q)^{-1}_\infty$ when $\zeta$ is a root of unity.  Prompted by a question of Stanley, a study of the polynomials $Q_n(\zeta)$ and their complex zeroes was undertaken in a series of papers by Boyer, Goh, Keith and Parry (see \cite{Boyergoh,Boyergoh2,Boyerkeith,Boyerparry,Parry}), and the functions $(\zeta q; q)^{-1}_\infty$ have also been studied in recent work of Bringmann, Craig, Males, and Ono \cite{BrigmannOnoMales} in the context of distribution of homology of Hilbert schemes and $t$-hook lengths.  Asymptotics for $Q_n(\zeta)$ were studied by Wright \cite{Wright} when $\zeta$ is any positive real number, and then by  Boyer and Goh \cite{Boyergoh,Boyergoh2} for $|\zeta|\neq 1$.  Our results essentially complete this study, proving asymptotics when $\zeta$ is any root of unity.  
 
 We use the circle method in the form of Parry \cite{Parry} as a template; however, significant technical effort is required to bound minor arcs that one does not encounter when $|\zeta|\neq 1$.  For example, we have to overcome the fact that when $\zeta$ is a root of unity, the series representation of the polylogarithm $\mathrm{Li}_s(\zeta)$ does not converge for $\mathrm{Re}(s) < 0$.

 Asymptotics in the case that $b \geq 5$ are as follows, and a general version with the sporadic cases $b \in \{2,3,4\}$ is stated in Theorem \ref{T:twistedetaproduct}.
\begin{theorem} \label{ThmB} Let $b \geq 5$. Then we have
\begin{align*}
    Q_n\left(\zeta_b\right) \sim \frac{\sqrt{1-\zeta_b}\mathrm{Li}_2\left(\zeta_b\right)^{\frac14}}{2\sqrt{\pi}n^{\frac34}}\exp\left( 2\sqrt{\mathrm{Li}_2\left(\zeta_b\right)}\sqrt{n}\right), \qquad n \to \infty.
\end{align*}
\end{theorem}

 \begin{remark} Our techniques for bounding minor arcs seem not to readily apply when $|\zeta|=1$ and $\zeta$ is not a root of unity, so we leave it as an open problem to prove asymptotic formulas for $Q_n(\zeta)$ in this case.  Boyer and Goh \cite{Boyergoh} prove that the unit circle is part of the {\it zero attractor} of the polynomials $Q_n(\zeta)$, and it follows from their results that asymptotic formulas for $Q_n(\zeta)$ cannot be uniform on any subset of $\mathbb{C}$ which is an open neighborhood of an arc of the unit circle.
 
 Furthermore, note that taking $\zeta \to 1$ in Theorem \ref{ThmB} does not recover the Hardy--Ramanujan asymptotic formula \eqref{partitionasy}.
 \end{remark}

The above results can be further developed to linear combinations of $p(a,b,n)$ (where $b$ is fixed). Each relevant combination $\sum_{0 \leq a < b} c_a p(a,b,n)$ corresponds to a polynomial \newline $P(x) := \sum_{0 \leq j < b} c_a x^a$ via $p(a,b,n) \mapsto x^a$. Perhaps the most combinatorially interesting cases are $c_a \in \{0, \pm1 \}$; i.e., differences of partition numbers. This means that for any two nonempty disjoint subsets $S_1, S_2 \subset [0, b-1]$ of integers, we consider the differences 
\begin{align*}
    \sum_{a \in S_1} p(a,b,n) - \sum_{a \in S_2} p(a,b,n).
\end{align*}

The asymptotic behavior of these differences is described in Theorem \ref{T:differenceswithsets}. When choosing the coefficients $c_a$ properly we can reduce the growth of the difference terms by canceling main terms. By doing so, we see that actually all formulas of Theorem \ref{ThmB} are required (as well as the sporadic cases described in Theorem \ref{T:twistedetaproduct}). In contrast, one can deduce formula \eqref{pequa} by only elementary means without a thorough analysis of the coefficients $Q_n(\zeta)$.

Moreover, for families of partition numbers, one can consider growth and sign changes simultaneously.  Of special interest here is the fact that any shift of $(S_1, S_2)$ to $(S_1+r,S_2+r)$ changes the ``phase" but not the ``amplitude" of the asymptotic terms. The following example demonstrates this fact.

\begin{example}\label{Ex:mod12combs} Let $b = 12$. For the sets $S_1 := \{1,2,5\}$ and $S_2 := \{0,3,4\}$, we find \newline $P_{S_1, S_2}(x) = x^5-x^4-x^3+x^2+x-1= (x-1)\Phi_{12}(x)$, where $\Phi_{12}(x)$ is the 12th cyclotomic polynomial. The corresponding difference of partition numbers is
\begin{align*}
    p(5,12,n) - p(4,12,n) - p(3,12,n) + p(2,12,n) + p(1,12,n) - p(0,12,n).
\end{align*}
By shifting the sets with integers $0 \leq r \leq 6$, we find that, more generally, all differences 
\begin{align*}
    \Delta_r(n) & := p(5+r,12,n) - p(4+r,12,n) - p(3+r,12,n)\\&\hspace{4cm}  + p(2+r,12,n) + p(1+r,12,n) - p(r,12,n)
\end{align*}
have the same growth in their amplitudes but have different phases of sign changes. Since $ P_{S_1, S_2}(x)$ has zeros $\{1, \pm \zeta_{12}, \pm \zeta^5_{12}\}$, the dominating term in the asymptotic expansion is induced by the root of unity $\zeta_6$. In fact, we obtain 
\begin{align} \label{Deltar}
    \frac{\Delta_r(n)}{An^{-\frac34}\exp\left(2\lambda_1 \sqrt{n}\right)} = \cos\left(\alpha - \frac{2\pi r}{6} + 2\lambda_2 \sqrt{n} \right) + o(1),
\end{align}
with $\lambda_1 + i\lambda_2 = \sqrt{\mathrm{Li}_2(\zeta_6)}$, where $A > 0$ and $\alpha \in [0,2\pi)$ are defined by 
$$Ae^{i\alpha} = \frac{\left(1 - i\sqrt{3}\right)^{\frac32}}{12\sqrt{2\pi}} \mathrm{Li}_2(\zeta_6)^{\frac{1}{4}}.$$
Note that $A$, $\alpha$, $\lambda_1$ and $\lambda_2$ all do not depend on the choice of $r$.  One can then asymptotically describe the regions of $n$, where 

$$
    p(5,12,n) + p(2,12,n) + p(1,12,n) > p(4,12,n) + p(3,12,n) + p(0,12,n), 
$$
and vice versa, using formula \eqref{Deltar}. Additionally, since $\mathrm{Re}(\sqrt{\mathrm{Li}_2(\zeta_6)}) < \mathrm{Re}(\sqrt{\mathrm{Li}_2(\zeta_{12})})$, we note with the help of Theorem \ref{ThmA} that
\begin{align*}
    \limsup_{n \to \infty} \left| \frac{p(a_1, 12, n) - p(a_2, 12, n)}{n^{-\frac{3}{4}}\exp\left(2\mathrm{Re}(\mathrm{Li}_2(\zeta_6))\sqrt{n}\right)}\right| = \infty.
\end{align*}
This implies some form of cancellation within the higher differences $\Delta_r(n)$, that ``exceeds" that of simple differences modulo $12$.  \end{example} 

In contrast to Example \ref{Ex:mod12combs}, when $b$ is prime, one cannot decrease the order of growth below simple differences using any rational combination of the $p(a,b,n)$.  This is shown in Section 3.

The paper is organized as follows. In Section 2, we collect some classic analytical tools and prove some key lemmas. This includes a careful study of the dilogarithm function $\mathrm{Li}_2(z)$ for values $|z|=1$. In Section 3, we state our main result, Theorem \ref{T:twistedetaproduct}, and applications.  We prove Theorem \ref{T:twistedetaproduct} in Sections 4 and 5 using the circle method.  Section 5 deals with the primary difficulty of bounding the minor arcs.

\section*{Acknowledgements}
We would like to thank Kathrin Bringmann, Joshua Males, Caner Nazaroglu, Ken Ono and Wadim Zudilin for useful discussions and for making comments on an earlier version of this paper.  We are grateful to the referees for their detailed comments that improved the exposition.

The first author is partially supported by the SFB/TRR 191 ``Symplectic Structures in Geometry, Algebra and Dynamics'', funded by the DFG (Projektnummer 281071066 TRR 191).

The second author is partially supported by the Alfried Krupp prize. 

\section{Preliminaries} 

We will need several analytical results in this work, which we collect in this section. Much of the items here are discussed in works such as \cite{Bringmannbook,BMJS,CaTir,IwanKow,Pinsky,Asymptbook}. The reader can skip this section and refer back to it as needed as we work through the proofs of our main theorems. 

\subsection{Classical asymptotic analysis and integration formulas}
A first tool is the well known Laplace's method for studying limits of definite integrals with oscillation, which we will use for evaluating Cauchy-type integrals.
\begin{theorem}[Laplace's method, see Section 1.1.5 of \cite{Pinsky}]\label{laplacemethod}
   Let $A,B: [a,b]\to \mathbb{C}$ be continuous functions. Suppose $x\neq x_0 \in [a,b]$ such that $\mathrm{Re}(B(x))<\mathrm{Re}(B(x_0)),$ and that 
   \begin{align*}
       \lim_{x\to x_0}\frac{B(x)-B(x_0)}{(x-x_0)^2} = -k\in \mathbb{C},
   \end{align*}
   with $\mathrm{Re}(k)>0.$ Then as $t \to \infty$
   \begin{align*}
       \int^b_{a}A(x)e^{tB(x)}dx = e^{tB(x_0)}\left(A(x_0)\sqrt{\frac{\pi}{tk}}+o\left(\frac{1}{\sqrt{t}}\right)\right).
   \end{align*}
\end{theorem}
We are ultimately interested in how the coefficients of a series $S(q):= \sum_{n\geq 0}a(n)q^n$ grow as $n\to \infty$. The classical Euler--Maclaurin summation formulas can be applied in many cases to link the growth of $a(n)$ to the growth of $S(q)$ as $q$ approaches the unit circle. We state the Euler--Maclaurin summation formulas in two different forms. 

\begin{theorem} [classical Euler--Maclaurin summation, see p. 66 of \cite{IwanKow}]\label{T:EulerMac}
    Let $\{x\}:=x - \lfloor x \rfloor$ denote the fractional part of $x$. For $N \in \mathbb{N}$ and $f:[1, \infty) \to \mathbb{C}$ a continuously differentiable function, we have
    $$
    \sum_{1 \leq n \leq N} f(n)=\int_1^N f(x)dx+\frac{1}{2}(f(N)+f(1))+\int_1^N f'(x)\left(\{x\}-\frac{1}{2}\right)dx.
    $$
\end{theorem}

We use a second formulation, due to Bringmann--Mahlburg--Jennings-Shaffer, to study the asymptotics of series with a complex variable approaching 0 within a fixed cone in the right-half plane.  It is not stated in \cite{BMJS}, but one can conclude from the proof that Theorem \ref{T:BMJS} is uniform in $0 \leq a \leq 1.$
  \begin{theorem}[uniform complex Euler--Maclaurin summation, \cite{BMJS}, Theorem 1.2] \label{T:BMJS}
      Suppose $0\leq \theta < \frac{\pi}{2}$ and suppose that $f:\mathbb{C}\to \mathbb{C}$ is holomorphic in a domain containing $\{re^{i\alpha}: |\alpha| \leq \theta\}$ with derivatives of sufficient decay; i.e., there is an $\varepsilon >0$ such that for all $m$ $f^{(m)}(z) \ll z^{-1-\varepsilon}$ as $z \to \infty$.  Then uniformly for $a \in [0,1]$, we have
      $$
      \sum_{\ell \geq 0} f(w(\ell+a))=\frac{1}{w}\int_0^{\infty} f(w)dw -\sum_{n=0}^{N-1} \frac{f^{(n)}(0)B_{n+1}(a)}{(n+1)!}w^n+O_N\left(w^{N}\right),
      $$
      uniformly as $w \to 0$ in $\Arg(w)\leq \theta$.
  \end{theorem}
  
  To identify a constant term in our asymptotic formula, we cite the following integral calculation of Bringmann, Craig, Males and Ono.
\begin{lemma}[\cite{BrigmannOnoMales}, Lemma 2.3]\label{L:Bringmannintegral}
For $a \in \mathbb{R}^+,$
\begin{align*}
    \int_0^{\infty} \left(\frac{e^{-ax}}{x(1-e^{-x})}-\frac{1}{x^2}-\left(\frac{1}{2}-a \right) \frac{e^{-ax}}{x} \right)dx =\log \left(\Gamma \left(a \right)\right)+\left(\frac{1}{2}-a\right)\log\left(a \right)-\frac{1}{2}\log(2\pi).
\end{align*}
\end{lemma}

Finally, we will use Abel partial summation extensively when bounding the twisted eta-products on the minor arcs.

  \begin{proposition}[Abel partial summation, see p. 3 of \cite{Tenenbaum}]\label{P:Abelpartialsummation}
Let $N \in \N_0$ and $M \in \N$. For sequences $\{a_n\}_{n \geq N}$, $\{b_n\}_{n \geq N}$ of complex numbers, if $A_n:=\sum_{N < m \leq n} a_m,$  then
  $$
  \sum_{N < n \leq N+M} a_nb_n=A_{N+M}b_{N+M}+\sum_{N < n < N+M} A_n(b_n-b_{n+1}).
  $$
  \end{proposition}

\subsection{Elementary bounds}
  The following bound for differences of holomorphic functions will be used in conjunction with Abel partial summation during the course of the circle method. The proof of the following is a straightforward application of the fundamental theorem of calculus and the maximum modulus principle.
\begin{lemma} \label{L:DifferenceEstimate} Let $f : U \to \C$ be a holomorphic function and $\overline{B_r(c)} \subset U$ a compact disk. Then, for all $a,b \in B_r(c)$ with $a \not= b$, we have 
\begin{align*}
    \left| f(b) - f(a)\right| \leq \max_{|z - c|=r} |f'(z)| |b-a|.
\end{align*}

\end{lemma}

We also need the following elementary maximum; for a proof see \cite{Parry}, Lemma 5.2.

\begin{lemma}\label{L:exponentialquotientmax}
For $a \in \left(-\frac{\pi}{2}, \frac{\pi}{2}\right)$ and $b \in \mathbb{R}$, we have
$$
\mathrm{Re}\left(\frac{e^{ia}}{1+ib} \right)\leq \cos^2\left(\frac{a}{2}\right),
$$
with equality if and only if $b$ satisfies $\Arg(1+ib)=\frac{a}{2}.$
\end{lemma}

\subsection{Bounds for trigonometric series and the polylogarithm}
We recall that for complex numbers $s$ and $z$ with $|z| < 1$ the polylogarithm $\mathrm{Li}_s(z)$ is defined by the series
\begin{align*}
    \mathrm{Li}_s(z) := \sum_{n \geq 1} \frac{z^n}{n^s}.
\end{align*}
We are especially interested in the case $s=2$, where $\mathrm{Li}_2(z)$ is called the dilogarithm. The appendix of a recent preprint of Boyer and Parry \cite{Boyerparry} contains many useful results on the dilogarithm, including the following key lemma.

\begin{lemma}[Proposition 4, \cite{Boyerparry}]\label{L:dilogdecreasing}
The function $\theta \mapsto \mathrm{Re} \left(\sqrt{\textnormal{Li}_2(e^{2\pi i \theta})}\right)$ is decreasing on the interval $[0,\frac12].$
\end{lemma}
We also need to consider the derivative of the function $\theta \mapsto \mathrm{Li}_2(e^{2\pi i\theta})$ and its partial sums. Let 

\begin{align*}
G_M(\theta) := \sum_{1 \leq m \leq M} \frac{e^{2\pi i\theta m}}{m}.
\end{align*}
Then the following bound holds. 
\begin{lemma} \label{L:CosSumBound} We have, uniformly for $0 < \theta < 1$, 
	\begin{align*}
	\left|G_M(\theta)\right| \ll \log\left(\frac{1}{\theta}\right) + \log\left( \frac{1}{1-\theta}\right), \qquad M \to \infty. 
	\end{align*}
	
	\end{lemma}

\begin{proof} Note that we have 
	\begin{align*}
	G_M(\theta) = \sum_{1 \leq m \leq M} \frac{\cos(2\pi \theta m)}{m} + i \sum_{1 \leq m \leq M} \frac{\sin(2\pi \theta m)}{m},
	\end{align*}
	and as it is well known that $\sum_{1 \leq m \leq M} \frac{\sin(2\pi \theta m)}{m}$ is uniformly bounded in $\mathbb{R}$ (see \cite{Stein} on p. 94) we are left with the cosine sum. Let $0 < \theta \leq \frac12$ and $M \in \mathbb{N}$. Consider the meromorphic function
	$$
	h_\theta(z) := \frac{\cos(2\pi \theta z)(\cot(\pi z)-i)}{2iz},
	$$
	together with the rectangle $R_M$ with vertices $\frac{1}{2}-iM,\frac{1}{2}+ \frac{i}{\theta}$, $\frac{i}{\theta} + M + \frac12$, and $-iM + M + \frac12$. Notice that in a punctured disk of radius $r<1$ centered at $k\geq 1$, we have  $$ 	h_\theta(z) = \frac{\cos(2\pi \theta k)}{2ik}\left(\frac{1}{\pi(z-k)}-i+O(z-k)\right).$$ With the residue theorem we find 
	\begin{align*}
	\oint_{\partial R_M} h_\theta(z) dz = 2\pi i \sum_{1 \leq m \leq M} \mathrm{Res}_{z=m} h_\theta(z) = \sum_{1 \leq m \leq M} \frac{\cos(2\pi \theta m)}{m},  
	\end{align*}
	where the contour is taken once in positive direction. By the invariance of the function under the reflection $\theta \to 1 - \theta$ it suffices to consider values $0 < \theta \leq \frac12$. A straightforward calculation shows that the bottom side integrals are bounded uniformly in $0 < \theta \leq \frac12$. Similarly, also using that $\cot(\pi z)$ is uniformly bounded on lines $\{\mathrm{Re}(z) \in \frac12 \Z \setminus \Z\}$ for the second part, one argues 
	\begin{align*} 
	\int_{M+\frac12-Mi}^{M+\frac12 + \frac{i}{\theta}} h_\theta(z)dz = O(1) + \int_{M+\frac12}^{M+\frac12 + \frac{i}{\theta}} h_\theta(z)dz \ll  \int_{\frac{1}{2}}^{\frac{1}{2}+\frac{i}{\theta}} \left| \frac{dz}{z+M}\right| \ll \log\left( \frac{1}{\theta}\right).     
	\end{align*}
 A similar argument yields
	\begin{align*}
	    \int_{\frac{1}{2}+\frac{i}{\theta}}^{\frac{1}{2} - Mi} h_\theta(z) dz \ll \log\left( \frac{1}{\theta} \right).
	\end{align*}
	
    For the rectangle top, note with the use of the substitution $z \mapsto \frac{z}{\theta}$ we have
	\begin{align} \label{RectTopSplit}
	\int_{M + \frac12 + \frac{i}{\theta}}^{\frac{1}{2}+\frac{i}{\theta}} h_\theta(z) dz = \int_{\theta \left(M + \frac12\right)+i}^{\frac{\theta}{2}+i} \frac{\cos(2\pi z) (\cot\left( \frac{\pi z}{\theta}\right)+i)}{z} dz-2i \int_{\theta \left(M + \frac12\right)+i}^{\frac{\theta}{2}+i}\frac{\cos(2\pi z)}{z} dz.
	\end{align}
	With the Residue Theorem we obtain
		\begin{align*}
	\int_{\frac{\theta}{2}+i}^{\theta(M+\frac12) + i} = \int_{\frac{\theta}{2}+i}^{i\infty} + \int_{\theta(M+\frac12) + i\infty}^{\theta(M+\frac12) + i}.
	\end{align*}
	and the bound $\cos(2\pi z)(\cot\left( \frac{\pi z}{\theta}\right) +i)= O(e^{-2\pi \mathrm{Im}(z)})$  as $\mathrm{Im}(z) \to \infty$ (since $0 < \theta \leq \frac12$) yields the uniform boundedness of the first integral in \eqref{RectTopSplit}. The integral on the right hand side of \eqref{RectTopSplit} can be bounded via standard contour integration, for instance, by using that $\int_\alpha^\infty \frac{\cos(x)}{x} dx \ll 1$ for all $\alpha \geq 1$. This completes the proof. 
\end{proof}

 Lemma \ref{L:CosSumBound} has the following important consequence.

\begin{lemma} \label{L:Gmaxbound} Let $a$, $b$, $h$ and $k$ be positive integers, such that $\gcd(a,b) = \gcd(h,k) = 1$. We assume that $a$ and $b$ are fixed. Then we have
\begin{align*}
    \sum_{\substack{1 \leq j \leq k \\ ak+bjh \not\equiv 0 \pmod{bk}}} \max_{m \geq 1} \left| G_m\left( \frac{a}{b} + \frac{hj}{k}\right) \right| = O(k),
\end{align*}
as $k \to \infty$, uniformly in $h$.
\end{lemma}
\begin{proof} The assumption $\gcd(h,k)=1$ implies that the map $j \mapsto hj$ is a bijection modulo $k$. One can show by elementary means that there is at most one $j_0$ such that $bk$ divides $ak+bj_0h$. First, we assume that such a $j_0$ exists. Using $G_m(x+1)=G_m(x)$, we first reorder the sum, then apply Lemma \ref{L:CosSumBound} to obtain 
\begin{align*}
  \sum_{\substack{1 \leq j \leq k \\ j \not= j_0}} \max_{m \geq 1} \left| G_m\left( \frac{a}{b} + \frac{hj}{k}\right) \right| = \sum_{1 \leq j \leq k-1} \max_{m \geq 1} \left| G_m\left( \frac{j}{k}\right) \right| &\ll \sum_{1 \leq j \leq k-1} \left(\log\left(\frac{k}{j}\right)+\log\left(\frac{k}{k-j}\right)\right) \\&= O(k)
\end{align*}
by Stirling's formula. Now, assume that there is no such $j_0$. In this case we find $0 < \alpha_{a,b,k} < b$, such that 
    \begin{align*}
        &\sum_{1 \leq j \leq k} \max_{m \geq 1} \left|G_m\left(\frac{a}{b}+\frac{hj}{k}\right)\right| =\max_{m \geq 1} \left|G_m\left(\frac{\alpha_{a,b,k}}{bk}\right)\right|+\max_{m \geq 1} \left|G_m\left(\frac{\alpha_{a,b,k}}{bk}+1-\frac{1}{k}\right)\right| \\&\hspace{6cm} + \sum_{1 \leq j \leq k-2} \max_{m \geq 1} \left|G_m\left(\frac{\alpha_{a,b,k}}{bk}+\frac{j}{k}\right)\right|.
    \end{align*}
    Using again Lemma \ref{L:CosSumBound} and Stirling's formula, we obtain
    \begin{align*}
    \sum_{1 \leq j \leq k-2} \max_{m \geq 1} \left|G_m\left(\frac{\alpha_{a,b,k}}{bk}+\frac{j}{k}\right)\right| &\ll \sum_{1 \leq j \leq k-2} \left(\log\left(\frac{k}{j}\right)+\log\left(\frac{k}{k-j}\right)\right) = O(k).
    \end{align*}
    Since clearly 
    $$\max_{m \geq 1} \left|G_m\left(\frac{\alpha_{a,b,k}}{bk}\right)\right| + \max_{m \geq 1} \left|G_m\left(\frac{\alpha_{a,b,k}}{bk}+1-\frac{1}{k}\right)\right| = O(\log (k)),$$ the lemma follows.
    \end{proof}

\section{Main results and applications}
\subsection{Twisted eta-products}

In this section, we record the general asymptotic formula for $Q_n(\zeta)$ where $\zeta$ is any root of unity.  Note that $\overline{Q_n(\zeta)}=Q_n(\overline{\zeta})$, so it suffices to find asymptotic formulas for $\zeta$ in the upper-half plane.  The following theorem of Boyer and Goh \cite{Boyergoh} regarding the dilogarithm distinguishes several cases in our main theorem.  Following \cite{Boyergoh}, we define
$$
\Psi_k(\theta):=\mathrm{Re}\left(\frac{\sqrt{\mathrm{Li}_2(e^{ik \theta})}}{k}\right), \qquad \text{for $0 \leq \theta \leq \pi,$  }
$$
and
$$0<\theta_{13}<\frac{2\pi}{3}<\theta_{23}<\pi$$
where each $\theta_{jk}$ is a solution to $\Psi_j\left(\theta\right)=\Psi_k\left(\theta\right)$.
Here,
$$
\theta_{13}=2.06672\dots , \quad \theta_{23}=2.36170\dots.
$$
Since the values $\theta_{13}$ and $\theta_{23}$ arise as solutions to a non-algebraic equation, it is very unlikely that they are rational multiples of $\pi$. Therefore, we will no longer consider them in future investigations. 
\begin{theorem}[\cite{Boyergoh}, discussion prior to Theorem 2]\label{T:BoyerGoh}
For $0\leq \theta \leq  \pi$, we have
$$
\max_{k \geq 1} \Psi_k(\theta)=\begin{cases} \Psi_1(\theta) & \text{if $\theta \in [0,\theta_{13}]$,} \\
\Psi_2(\theta) & \text{if $\theta \in [\theta_{23},\pi]$,}
\\
\Psi_3(\theta) &  \text{if $\theta \in  [\theta_{13},\theta_{23}]$.}\end{cases}
$$
\end{theorem}
Following \cite{Parry}, define 
$$
\omega_{h,k}(z):=\prod_{j=1}^k \left(1-z \zeta_k^{-jh}\right)^{\frac{j}{k}-\frac{1}{2}}.
$$
We have the following asymptotic formulas.
\begin{theorem}\label{T:twistedetaproduct}

\textup{(1)} If $2\pi \frac{a}{b}\in (0,\theta_{13}),$ then
$$
Q_n\left(\zeta_b^a\right) \sim \frac{\sqrt{1-\zeta_b^a}\mathrm{Li}_2\left(\zeta_b^a\right)^{\frac14}}{2\sqrt{\pi}n^{\frac34}}\exp\left({2\sqrt{\mathrm{Li}_2\left(\zeta_b^a\right)}\sqrt{n}}\right), \qquad n \to \infty.
$$
\textup{(2)} If $2\pi \frac{a}{b}\in (\theta_{23},\pi),$ then
$$
Q_n\left(\zeta_b^a\right) \sim \frac{(-1)^n\sqrt{1-\zeta_b^a}\mathrm{Li}_2\left(\zeta_b^{2a}\right)^{\frac14}}{2\sqrt{2\pi}n^{\frac34}}\exp\left({\sqrt{\mathrm{Li}_2\left(\zeta_b^{2a}\right)}\sqrt{n}}\right), \qquad n \to \infty.
$$
\textup{(3)} If $2\pi \frac{a}{b}\in (\theta_{13},\theta_{23})\setminus \left\{\frac{2\pi}{3}\right\},$ then
$$
Q_n\left(\zeta_b^a\right) \sim (\zeta_3^{-n}\omega_{1,3}(\zeta_b^{a})+\zeta_3^{-2n}\omega_{2,3}(\zeta_b^{a}))\frac{\mathrm{Li}_2(\zeta_b^{3a})^{\frac14}}{2\sqrt{3\pi}n^{\frac34}}\exp\left(\frac{2}{3}\sqrt{\mathrm{Li}_2(\zeta_b^{3a})}\sqrt{n}\right), \qquad n \to \infty.
$$
\textup{(4)}  We have
$$
   Q_n\left( \zeta_3\right)\sim \frac{\zeta_3^{-2n}(1- \zeta_3^2)^{\frac16}(1- \zeta_3)^{\frac12}\Gamma\left(\frac{1}{3}\right)}{2(6\pi n)^{\frac23}}\exp\left( \frac{2\pi}{3}\sqrt{\frac{n}{6}}\right), \qquad n \to \infty.
$$
\end{theorem}

We prove Theorem \ref{T:twistedetaproduct} in Sections 4 and 5.

\begin{remark}
Recall that we have the asymptotic formula \eqref{partitionasy} for $Q_n(1)=p(n)$, whereas for $Q_n(-1)$, standard combinatorial methods give
$$
\sum_{n \geq 0} Q_n(-1)q^n=\frac{1}{(-q;q)_{\infty}}=\frac{\left(q;q\right)_{\infty}}{(q^2;q^2)_{\infty}}=\left(q;q^2\right)_{\infty}=\sum_{n \geq 0} (-1)^np_{\mathcal{DO}}(n)q^n,
$$
where $p_{\mathcal{DO}}(n)$ counts the number of partitions of $n$ into distinct odd parts.  An asymptotic formula for $p_{\mathcal{DO}}(n)$ can be worked out using standard techniques. For example, Ingham's Tauberian theorem (see \cite{BMJS}, Theorem 1.1) with the modularity of the Dedekind $\eta$-function yields
$$
Q_n(-1)=(-1)^np_{\mathcal{DO}}(n) \sim \frac{(-1)^n}{2(24)^{\frac14}n^{\frac34}}\exp\left(\pi \sqrt{\frac{n}{6}}\right).
$$
Note the lack of uniformity in the asymptotic formulas for $\zeta$ near $\pm 1$; in particular, the asymptotic formulas for $Q_n(1)$ and $Q_n(-1)$ cannot be obtained by taking $\frac{a}{b} \to \pm 1$ in cases (1) and (2).
\end{remark}

\subsection{Applications to differences of partition functions}
In this section, we apply Theorem \ref{T:twistedetaproduct} to differences of the partition functions $p(a,b,n)$, partitions of $n$ with number of parts congruent to $a$ modulo $b$.  The following elementary proposition relates the the numbers $p(a,b,n)$ to the coefficients $Q_n( \zeta_b^{j})$.
    
    \begin{proposition}\label{P:pabnrewrite}
    We have
    $$
    p(a,b,n)=\frac{1}{b}\sum_{0 \leq j \leq b-1} \zeta_b^{-ja}Q_n\left( \zeta_b^{j}\right).
    $$
    \end{proposition}
    \begin{proof}
    Using orthogonality, we can write the indicator functions of congruence classes as $$1_{x \equiv a \hspace{-.3cm} \pmod{b}}= \frac{1}{b}\sum_{j=0}^{b-1} \zeta_b^{-ja} \zeta_b^{jx},$$ and by standard combinatorial techniques (see \cite{andrewsbook}, Ch. 1) one has $$Q_n(z)=\sum_{\lambda \vdash n} z^{\ell(\lambda)}.$$  Hence,
    \begin{align*}
     p(a,b,n)=\sum_{\lambda \vdash n} \frac{1}{b}\sum_{0 \leq j \leq b-1}  \zeta_b^{-ja} \zeta^{j\ell(\lambda)}_{b}&=\frac{1}{b}\sum_{0 \leq j \leq b-1}  \zeta_b^{-ja}\sum_{\lambda \vdash n}  \left(\zeta_b^{j}\right)^{\ell(\lambda)} =\frac{1}{b}\sum_{0 \leq j \leq b-1}  \zeta_b^{-ja} Q_n\left( \zeta_b^{j}\right),
    \end{align*}
    which completes the proof. 
    \end{proof}
Thus, asymptotics for differences of $p(a,b,n)$ can be identified using the asymptotic formulas found for $Q_n(\zeta)$ in Theorem \ref{T:twistedetaproduct}, in particular the equidistribution of the largest part in congruence classes follows immediately: $\frac{p(a,b,n)}{p(n)} \sim \frac{1}{b}.$  Furthermore, using the Hardy--Ramanujan--Rademacher exact formula for $p(n)$ (\cite{andrewsbook}, Theorem 5.1) with Theorem \ref{T:twistedetaproduct}, one can improve this in the various cases.  For example, for $b \geq 5,$ we have
$$
\lim_{n \to \infty}\left(\frac{b\cdot p(a,b,n) - \frac{2\sqrt{3}}{24n-1}\exp\left(\frac{\pi\sqrt{24n-1}}{6}\right)\left(1-\frac{6}{\pi\sqrt{24n-1}} \right)}{An^{-\frac34}\exp(2\lambda_1\sqrt{n})}
-\cos\left(\alpha-\frac{2\pi a}{b}+2\lambda_2\sqrt{n}\right)\right)=0,
$$
where $\lambda_1+i\lambda_2=\sqrt{\mathrm{Li}_2( \zeta_b)},$ $A \geq 0,$ and $\alpha \in [0,2\pi)$ are defined by
$$
Ae^{i\alpha}=\sqrt{\frac{(\lambda_1+i\lambda_2)(1- \zeta_b^a)}{\pi}}.
$$

We record a number of results below in the same vein.  Considering simple differences, $p(a_1,b,n)-p(a_2,b,n),$ it follows from Proposition \ref{P:pabnrewrite} that when rewriting each $p(a,b,n)$ in terms of $Q_n$ the two summands $Q_n(1)=p(n)$ cancel.  With a few exceptions for $b=4$, the dominant terms are always $Q_n(\zeta_b)$ and $Q_n(\zeta_b^{b-1})$.

\begin{theorem}\label{T:simpledifferences}
Let $0 \leq a_1 < a_2 \leq b-1.$ \\
\textup{(1)} For $b=2,$ we have $p(0,2,n)-p(1,2,n)=Q_n(-1).$ \\
\textup{(2)}  For $b=3$, we have 
    $$
    \frac{p(a_1,3,n)-p(a_2,3,n)}{A_{a_1,a_2}n^{-\frac23}\exp\left(\frac{2\pi}{3}\sqrt{\frac{n}{6}}\right)} = \cos\left(\alpha_{a_1,a_2}-\frac{4\pi n}{3}\right) + o(1),
    $$
    where $A_{a_1,a_2} \geq 0$ and $\alpha_{a_1,a_2} \in [0,2\pi)$ are defined by
    $$
    A_{a_1,a_2}e^{i\alpha_{a_1,a_2}}=( \zeta_3^{-a_1}- \zeta_3^{-a_2})(1- \zeta_3^{2})^{\frac16}(1- \zeta_3)^{\frac12}\frac{\Gamma\left(\frac{1}{3}\right)}{3(6\pi)^{\frac23}}.
    $$
    \textup{(3)} For $b=4,$ and $a_1,a_2$ of opposite parity, we have $$p(a_1,4,n)-p(a_2,4,n) \sim \frac{(-1)^{a_1} - (-1)^{a_2}}{4}Q_n(-1).$$
    \textup{(4)} For $b \geq 5$, or for $b=4$ and $a_1,a_2$ of the same parity, we have
$$
\frac{p(a_1,b,n)-p(a_2,b,n)}{B_{a_1,a_2,b}n^{-\frac34}\exp(2\lambda_1\sqrt{n})}=\cos\left(\beta_{a_1,a_2,b}+2\lambda_2\sqrt{n} \right) + o(1),
$$
where $\lambda_1+i\lambda_2=\sqrt{\mathrm{Li}_2( \zeta_b)}$, $B_{a_1,a_2,b} \geq 0$ and $\beta_{a_1,a_2,b} \in [0,2 \pi)$ are defined by
    $$
    B_{a_1,a_2,b}e^{i\beta_{a_1,a_2,b}}=\frac{\left( \zeta_b^{-a_1}- \zeta_b^{-a_2}\right)}{b}\sqrt{\frac{(1- \zeta_b)(\lambda_1+i\lambda_2)}{\pi}}.
    $$

\end{theorem}

\begin{proof}
When $b=2$, Theorem \ref{T:simpledifferences} follows from Proposition \ref{P:pabnrewrite}.  Let $b \geq 5.$  By Proposition \ref{P:pabnrewrite}, it follows that
$$
p(a_1,b,n)-p(a_2,b,n)=\frac{1}{b}\mathrm{Re}\left(  (\zeta_b^{-a_1} - \zeta_{b}^{-a_2})Q_n( \zeta_b)\right)+\frac{1}{b}\sum_{1 \leq j \leq b-2} (\zeta_b^{-ja_1} - \zeta_b^{-ja_2})Q_n\left( \zeta_b^{j}\right).
$$
Upon dividing both sides by $B_{a_1,a_2,b} n^{-\frac34} \exp(2\lambda_1\sqrt{n})$, Lemma \ref{L:dilogdecreasing} implies that the sum on the right is $O(e^{-c\sqrt{n}})$, for some $c>0,$ and it then follows from Theorem \ref{T:twistedetaproduct} that
\begin{align} \label{analogforfuture}
\frac{p(a_1,b,n)-p(a_2,b,n)}{B_{a_1,a_2,b}n^{-\frac34}\exp\left(2\lambda_1\sqrt{n}\right)}=\mathrm{Re}\left(e^{i\left(\beta_{a_1,a_2,b}+2\lambda_2\sqrt{n}\right)}\right)+o(1)+o\left(e^{-c\sqrt{n}} \right),
\end{align}
which gives the claim of Theorem \ref{T:simpledifferences}.  

The other cases are proved similarly by noting that Lemma \ref{L:dilogdecreasing} and Theorem \ref{T:twistedetaproduct} imply that the $Q_n( \zeta_b)$ and $Q_n( \zeta_b^{-1})$ terms always dominate $Q_n( \zeta_b^{j})$ for $j \neq \pm 1,$ except in the case that $b=4$, where $Q_n(-1)$ dominates when $a_1,a_2$ have opposite parity.  But $Q_n(-1)$ vanishes from the sum in Proposition \ref{P:pabnrewrite} when $b=4$ and $a_1,a_2$ have the same parity, which leads to the third case in Theorem \ref{T:simpledifferences}.
\end{proof}

More generally, if $P_{\bold{v}}(x):=\sum_{0 \leq a \leq b-1} v_ax^a \in \mathbb{R}[x]$ with $\bold{v}:= (v_0,...,v_{b-1})\in \mathbb{R}^b$, then Theorem \ref{T:twistedetaproduct} implies asymptotic formulas for any weighted count
$
\sum_{a=0}^{b-1}v_ap(a,b,n).
$
To state our general theorem, we let

\begin{multicols}{2}
$$
L(e^{i\theta}):=\begin{cases}\sqrt{\mathrm{Li}_2(e^{i\theta})} & \text{if $0 \leq \theta < \theta_{13}$,} \\
\frac{\sqrt{\mathrm{Li}_2(e^{3i\theta})}}{3} & \text{if $\theta_{13} < \theta < \theta_{23}$,} \\
\frac{\sqrt{\mathrm{Li}_2(e^{2i\theta})}}{2} & \text{if $\theta_{23} < \theta \leq \pi$.} \end{cases}
$$
Let $\mathcal{Z}(P_{\bold{v}})$ be the roots of $P_{\bold{v}}$, and let \newline $\lambda_1+i\lambda_2=L( \zeta_b^{a_0})$ for $a_0 \leq \frac{b}{2}$, whenever
$$
\mathrm{Re}\left(L( \zeta_b^{a_0})\right)=\max_{\substack{0 \leq a \leq \frac{b}{2} \\  \zeta_b^a\notin \mathcal{Z}(P_{\bold{v}})}} \mathrm{Re} \left(L( \zeta_b^a)\right).
$$
\begin{figure}[H]
    \centering
    \includegraphics[width = 70mm]{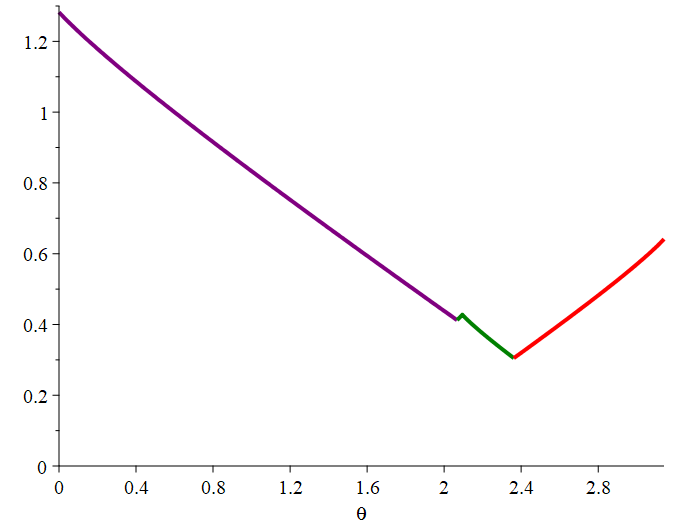} \caption{$\mathrm{Re}(L(e^{i\theta}))$ for $0\leq \theta \leq \pi.$}
    \label{fig:my_label}
\end{figure}
\end{multicols}

\begin{theorem}\label{T:genericrealsums}
With notation as above, we have the following asymptotic formulas. \\
    \textup{(1)} If $a_0=0,$ then $\sum_{0 \leq a \leq b-1} v_ap(a,b,n) \sim \frac{P_{{\bf v}}(1)}{b}p(n).$ \\
    \textup{(2)} If $0 < 2 \pi \frac{a_0}{b} < \theta_{13},$ then
    $$
     \frac{\sum_{0 \leq a \leq b-1}v_ap(a,b,n)}{A_{a_0,b,\bold{v}} n^{-\frac{3}{4}} \exp(2\lambda_1\sqrt{n})} = \cos\left(\alpha_{a_0,b,\bold{v}}+2\lambda_2\sqrt{n} \right) + o(1),
    $$
    where $A_{a_0,b,\bold{v}} \geq 0$ and $\alpha_{a_0,b,\bold{v}} \in [0,2\pi)$ are defined by
    $$
    A_{a_0,b,\bold{v}} \cdot e^{i \alpha_{a_0,b,\bold{v}}}= \frac{P_{\bold{v}}( \zeta_b^{-a_0})}{b}\sqrt{\frac{(\lambda_1+i\lambda_2)(1- \zeta_b^{a_0})}{\pi}}.
    $$
    \textup{(3)} If $\theta_{13} < 2 \pi \frac{a_0}{b} < \theta_{23}$ and $a_0 \neq \frac{b}{3},$ then
    \begin{align*}
    & \frac{\sum_{0 \leq a \leq b-1}v_ap(a,b,n)}{A_{a_0,b,\bold{v}} n^{-\frac{3}{4}} \exp(2\lambda_1\sqrt{n})} = B_{a_0,b}\cos\left(\alpha_{a_0,b,\bold{v}}+\beta_{a_0,b}-\frac{2\pi n}{3} + 2 \lambda_2\sqrt{n}\right) \\ & \hspace{3cm} + C_{a_0,b}\cos\left(\alpha_{a_0,b,\bold{v}}+\gamma_{a_0,b}-\frac{4\pi n}{3} + 2 \lambda_2\sqrt{n}\right) + o(1),
    \end{align*}
    where $B_{a_0,b},C_{a_0,b} \geq 0$ and $\beta_{a_0,b},\gamma_{a_0,b} \in [0,2\pi)$ are defined by
    $$
    B_{a_0,b}e^{i\beta_{a_0,b}} = \omega_{1,3}( \zeta_b^{a_0}),
    $$
    $$
    C_{a_0,b}e^{i\gamma_{a_0,b}}= \omega_{2,3}( \zeta_b^{a_0}).
    $$
    \textup{(4)} If $a_0=\frac{b}{3}$, then
    $$
    \frac{b\sum_{0 \leq a \leq b-1}v_ap(a,b,n)}{D_{\bold{v}} n^{-\frac23} \exp\left(\frac{2\pi}{3}\sqrt{\frac{n}{6}}\right)} = \cos\left(\delta_{\bold{v}}-\frac{4\pi n}{3}\right) + o(1),
    $$
    where $D_{\bold{v}} \geq 0$ and $\delta_{\bold{v}} \in [0,2\pi)$ are defined by
    $$
    D_{\bold{v}} \cdot e^{i\delta_{\bold{v}}}=\frac{(1- \zeta_3^{2})^{\frac16}(1- \zeta_3)^{\frac12}\Gamma(\frac13)}{2(6\pi)^{\frac23}}P_{\bold{v}}\left( \zeta_3^{-1}\right).
    $$
    \textup{(5)} If $\theta_{23} < 2 \pi \frac{a_0}{b} < \pi$, then
    $$
     \frac{\sum_{0 \leq a \leq b-1}c_vp(a,b,n)}{A_{a_0,b,\bold{v}} n^{-\frac34}\exp\left(2\lambda_1\sqrt{n}\right)} = \cos\left(\alpha_{a_0,b,\bold{v}}+\pi n+2\lambda_2\sqrt{n} \right) + o(1).
    $$
    \textup{(6)} If $a_0=\frac{b}{2},$ then $\sum_{0 \leq a \leq b-1} (-1)^a v_ap(a,b,n) \sim \frac{P_{\bold{v}}(-1)}{b}Q_n(-1).$
\end{theorem}
\begin{proof}
     Theorem \ref{T:genericrealsums} is proved similarly to Theorem \ref{T:simpledifferences}. For Cases (1) and (6), the asymptotic formula for $Q_n(1) \sim p(n)$ and Proposition \ref{P:pabnrewrite} directly implies
     \begin{align*}\sum_{0 \leq a \leq b-1} v_ap(a,b,n) &\sim \frac{1}{b}(v_0Q_n(1) + v_1Q_n(1)+...+v_{b-1}Q_n(1)) = \frac{P_{\textbf{v}}(1)p(n)}{b},\\ \sum_{0 \leq a \leq b-1}(-1)^a v_ap(a,b,n) &\sim \frac{1}{b}(v_0Q_n(-1) - v_1Q_n(-1)+...+(-1)^{b-1}v_{b-1}Q_n(-1)) \\ &= \frac{P_{\textbf{v}}(-1)Q_n(-1)}{b},
     \end{align*}
for Cases (1) and (6) respectively.  
   \par For Cases (2), (3), and (5), the asymptotic main term (using that $\overline{Q_n(\zeta)} = Q_n(\overline{\zeta})$) is
    $$p(a,b,n) \sim \frac{1}{b}\left(\zeta^{-a_0 a}_bQ_n(\zeta^{a_0}_b)+\overline{\zeta^{-a_oa}_b{Q_n(\zeta^{a_0}_b)}}\right) = \frac{2}{b}\mathrm{Re}\left(\zeta^{-a_0 a}_bQ_n(\zeta^{a_0}_b)\right). $$
    The proof then follows by applying Theorem \ref{T:twistedetaproduct} and dividing by the appropriate normalizing factors in analog to Equation \eqref{analogforfuture}.
    \par For case (4), there is only one main term in the sum for $$p(a,b,n)\sim \zeta^{-\frac{a}{3}}Q_n\left(\zeta^{\frac{1}{3}}\right).$$ Applying Theorem \ref{T:twistedetaproduct} and the analog to Equation \eqref{analogforfuture} again proves the claim.
\end{proof}

One application of the above theorem generalizes Theorem \ref{T:simpledifferences} to differences of partitions with number of parts modulo $b$ in one of two disjoint sets of residue classes $S_1, S_2 \subset [0,b-1]$.  That is, we consider
$$
P_{S_1,S_2}(x):=\sum_{a \in S_1} x^a-\sum_{a \in S_2} x^a, 
$$
and prove a more explicit version of Theorem \ref{T:genericrealsums} in this case.  Since $P_{S_1,S_2}(x)$ is a polynomial of degree at most $b-1$ with integer coefficients, there must exist some $d \mid b$ such that $ \zeta_d \notin \mathcal{Z}(P_{S_1,S_2}),$ otherwise $P_{S_1,S_2}$ would be divisible by $\prod_{d \mid b} \Phi_{d}(x)=x^b-1,$ where $\Phi_d$ is the $d$-th cyclotomic polynomial, a contradiction.

\begin{theorem}\label{T:differenceswithsets}
Let $b \geq 2$ and let $S_1, S_2 \subset [0,b-1]$ be disjoint subsets of integers.  If $|S_1| \neq  |S_2|,$ then 
$$\sum_{a \in S_1}p(a,b,n)-\sum_{a \in S_2}p(a,b,n) \sim \frac{(|S_1|-|S_2|)p(n)}{b}.$$  
Otherwise, if $|S_1|=| S_2|$, then we have the following cases.  Let $d_0$ be the largest integer such that $d_0 \mid b$ and $ \zeta_{d_0} \notin \mathcal{Z}(P_{S_1,S_2}).$ \\
    \textup{(1)} If $d_0\geq 5$ or if $d_0=4$ and $-1 \in \mathcal{Z}(P_{S_1,S_2})$, then
    $$
    \frac{\sum_{a \in S_1} p(a,b,n)-\sum_{a \in S_2} p(a,b,n)}{A_{d_0,S_1,S_2} n^{-\frac34} \exp(2\lambda_1\sqrt{n})} = \cos\left(\alpha_{d_0,S_1,S_2}+2\lambda_2\sqrt{n}\right) + o(1),
    $$
    where $\lambda_1+i\lambda_2=\sqrt{\mathrm{Li}_2( \zeta_{d_0})}$, $A_{d_0,S_1,S_2}\geq 0$ and $\alpha_{d_0,S_1,S_2} \in [0,2\pi)$ are defined by
    $$
    A_{d_0,S_1,S_2} \cdot e^{i\alpha_{d_0,S_1,S_2}}=\frac{P_{S_1,S_2}\left( \zeta_{d_0}^{-1}\right)}{b}\sqrt{\frac{(\lambda_1+i\lambda_2)(1- \zeta_{d_0})}{\pi}}.
    $$
    \textup{(2)}  If $d_0=4$ or $3$ with $b$ even, and $-1 \notin \mathcal{Z}(P_{S_1,S_2})$, we have the following allowable sets:
    \begin{itemize}
\item If $d_0 = 4$ and $(k_1+k_2)\not\equiv 0 \pmod{2}$,
    \begin{align*}
     S_1\times S_2&= \{a: a\leq b-1, \; a\equiv k_1 \pmod{4}\}\times\{a: a\leq b-1, \; a\equiv k_2 \pmod{4}\}.
    \end{align*}
 
\item If $d_0=3$, $b$ is even and $k_1 \not\equiv k_2 \pmod{2},$
\begin{align*}
    S_1\times S_2 &= \{a: a\leq b-1,\; a\equiv k_1 \pmod{2}\}\times\{a: a\leq b-1,\; a\equiv k_2 \pmod{2}\}.
\end{align*}

\end{itemize}
The asymptotic formulas are then given by
    $$
    \sum_{a \in S_1}p(a,b,n)-\sum_{a \in S_2} p(a,b,n) \sim \frac{N_{S_1,S_2}}{b}Q_n(-1),$$
    where
    $$
    N_{S_1,S_2}=\begin{cases} (-1)^{k_1}b& \textnormal{if}\;d_0=3,\; b \; \textnormal{even},\\  (-1)^{k_1}\frac{b}{2} & \textnormal{if}\;d_0=4.\end{cases}
    $$
    \textup{(3)}  If $d_0=3$, $b$ is even and $-1 \in \mathcal{Z}(P_{S_1,S_2})$, or $d_0=3$ and $b$ is odd, we have the following sets $S_1$ and $S_2$:
    \begin{itemize}
        \item If $b$ is odd, $S_1$ and $S_2$ must contain distinct residue classes modulo $3$.
        \item If $b$ is even, 
        \begin{align*}
            S_1\times S_2 &= \{a: a\leq b-1,\; a\equiv k_1 \;\textnormal{or}\; k_2 \pmod{6}\}\times\{a: a\leq b-2,\; a\equiv k_3\; \textnormal{or}\; k_4 \pmod{6}\},
        \end{align*}
        where 
\begin{align*}
    (k_1,k_2,k_3,k_4) &= (0,3,2,5),\; (2,5,0,3),\; (1,4,2,5),\; (2,5,1,4),
    (1,3,1,4),\; \textnormal{or}\;  (1,4,1,3).
\end{align*}
    \end{itemize}
The asymptotic formula is then given by
    $$
    \sum_{a \in S_1} p(a,b,n)-\sum_{a \in S_2} p(a,b,n) \sim \frac{2}{b}\mathrm{Re}\left(Q_n(\zeta_3)P_{S_1,S_2}\left(\zeta^{-1}_3\right)\right). $$
\textup{(4)}  If $d_0=2,$ then for some $a_1, a_2$ of opposite parity we have 
$$ (S_1,S_2)=\{a \equiv a_1 \pmod{2}\} \times \{a \equiv a_2 \pmod{2}\}$$
 and
    $$
    \sum_{a \in S_1}p(a,b,n)-\sum_{a \in S_2} p(a,b,n) \sim (-1)^{a_1}Q_n(-1).$$  
  
\end{theorem}

\begin{proof}[Proof sketch of Theorem \ref{T:differenceswithsets}]
The asymptotic analysis is similar to the proof of Theorem \ref{T:simpledifferences}, since Lemma \ref{L:dilogdecreasing} and Theorem \ref{T:twistedetaproduct} imply that the sequence $Q_n( \zeta_d)$, ranked from in asymptotic order from least to greatest, is
$$
Q_n( \zeta_3), \ Q_n(i), \ Q_n(-1), \ Q_n(\zeta_5), \ Q_n(\zeta_6), \dots, Q_n(1)=p(n).
$$
Thus, for example in case (3), the asymptotic behavior of $\sum_{a \in S_1} p(a,b,n) - \sum_{a \in S_2} p(a,b,n)$ is determined by $Q_n( \zeta_3)$ and $Q_n( \zeta_3^{-1})$ since all other $Q_n( \zeta_b^a)$ vanish in $P_{S_1,S_2}( \zeta_{b}^{-j})Q_n( \zeta_b^{j})$.  Furthermore, in this case we must have
$$
\frac{x^{b}-1}{\Phi_3(x)} \mid P_{S_1,S_2}(x),
$$
and $\deg (P_{S_1,S_2}(x))\leq b-1$.  Since
$$
\frac{x^{b}-1}{\Phi_3(x)}=x^{b-2}-x^{b-3}+x^{b-5}-x^{b-6}+ \dots + x^{4}-x^3+x-1,
$$
this leads directly to the possible sets $S_1$ and $S_2$ described in case 3.  Finally, note that, given $S_1$ and $S_2$ defined in terms of $a_1$ and $a_2$ modulo $b$, we have
$$
P_{S_1,S_2}\left( \zeta_3^{-1}\right)=\sum_{\substack{0 \leq a \leq b \\a \equiv a_1 \pmod{3}}}  \zeta_{3}^{-a}-\sum_{\substack{0 \leq a \leq b \\ a \equiv a_2 \pmod{3}}}  \zeta_{3}^{-a}=\frac{b}{3}( \zeta_3^{-a_1}- \zeta_3^{-a_2}).
$$  This is taken into account in the definition of the constant $B_{S_1,S_2}.$ The other cases are proved similarly.
\end{proof}

We make a few remarks. 
\begin{remark} \label{R:differenceremark}

\textup{(1)} Theorem \ref{T:differenceswithsets} case (1), in combination with Lemma \ref{L:dilogdecreasing}, shows that one can reduce the growth of the amplitudes in the differences exponentially, as long as the corresponding polynomial $P_{S_1, S_2}$ vanishes at the crucial roots of unity. But the options for such types of cancellation strongly depend on $b$. For example, if $b$ is a prime number, there is not even a rational combination (except for the trivial combination) such that the amplitudes of $\sum_{0 \leq a < b} v_a p(a,b,n)$ grow exponentially less than any simple difference $p(a_1,b,n) - p(a_2,b,n)$. The simple algebraic reason behind this is that the minimal polynomial of $\zeta_b$ has degree $b-1$ in this case. It would be interesting to find a purely combinatorial interpretation for this fact. \\

\textup{(2)} Note that if we shift the residue classes in $S_1$ and $S_2$ by some integer $r$, and then take least residues modulo $b$ to compute $P_{S_1+r,S_2+r}(x)$, then this polynomial has the same roots of unity as $P_{S_1,S_2}(x)$, and so 
$$
\sum_{a \in S_1+r}p(a,b,n)-\sum_{a \in S_2+r}p(a,b,n)
$$
always has the same asymptotic behavior in Theorem \ref{T:differenceswithsets} as $\sum_{a \in S_1}p(a,b,n)-\sum_{a \in S_2}p(a,b,n).$ At the same time, the phase in the cosine changes. Indeed, we obtain 
\begin{align} \label{E:phaseshift}
        \frac{\sum_{a \in S_1+r} p(a,b,n)-\sum_{a \in S_2+r} p(a,b,n)}{A_{d_0,S_1,S_2} n^{-\frac34}\exp\left(2\lambda_1\sqrt{n}\right)} = \cos\left(\alpha_{d_0,S_1,S_2} - \frac{2\pi r}{d_0} +2\lambda_2\sqrt{n}\right) + o(1),
\end{align}
where all the constants are the same as in Theorem \ref{T:differenceswithsets} (1).  One can further use trigonometric identities to obtain a wider class of more classical asymptotic formulas, for instance regarding squared partition differences. Indeed, if $4|d_0$ in \eqref{E:phaseshift}, one finds using $\sin^2(x) + \cos^2(x) = 1$, as $n \to \infty$,
\begin{align*}
    & \left( \sum_{a \in S_1} p(a,b,n)-\sum_{a \in S_2} p(a,b,n) \right)^2 + \left( \sum_{a \in S_1+\frac{d_0}{4}} p(a,b,n)-\sum_{a \in S_2+\frac{d_0}{4}} p(a,b,n) \right)^2  \\
    & \hspace{8.5cm} \sim A_{d_0,S_1,S_2}^2 n^{-\frac32}\exp\left(4\lambda_1\sqrt{n}\right).
\end{align*}

\textup{(3)} Cases 3 and 4 in Theorem \ref{T:differenceswithsets} show that for $d_0\in\{2,3\}$, there are finitely many sets $S_1$ and $S_2$ such that the asymptotic behavior of  $\sum_{a \in S_1}p(a,b,n)-\sum_{a \in S_2}p(a,b,n)$ is determined by $Q_n( \zeta_{d_0})$, and that the number of such sets is independent of $b$.  In fact this is true for any $d_0$, and we leave it as an open problem to describe the sets $S_1,S_2$ in general.  The sets $S_1$ and $S_2$ need not each consist of one congruence class each modulo $d_0,$ for example with $b=10$, the polynomial $P_{\{1,3,6,8\},\{0,2,5,7\}}(x)$ has $d_0=5.$ 
\end{remark}

\subsection{Examples} In this section, we provide some examples. 

\begin{example} Let $b=6$, and consider the difference $p(1,6,n) - p(5,6,n)$. Then according to Theorem \ref{T:simpledifferences}, we obtain
\begin{align} \label{E:p6explicit}
    \frac{p(1,6,n) - p(5,6,n)}{Bn^{-\frac34} \exp\left(2\lambda_1 \sqrt{n}\right)} = \cos\left( \beta + 2\lambda_2 \sqrt{n}\right) + o(1), 
\end{align}
where $\lambda_1 + i\lambda_2 = \sqrt{\mathrm{Li}_2(\zeta_6)}$, and $B > 0$ and $\beta \in [0, 2\pi)$ are given implicitely by 
\begin{align*}
    Be^{i\beta} = \frac{\zeta_6^{-1} - \zeta_6^{-5}}{6} \sqrt{\frac{(1-\zeta_6) (\lambda_1 + i \lambda_2)}{\pi}} = \frac{i}{\sqrt{12 \pi}} \sqrt{\left( \frac12 - i\frac{\sqrt{3}}{2}\right) \sqrt{\mathrm{Li}_2(\zeta_6)}}.  
\end{align*}
Note that this implies (choosing $B$ to be the absolute value of the right hand side of the equation above)
\begin{align}
    \label{E:lambda1} \lambda_1 & = 0.81408 \ldots, \qquad \lambda_2 = 0.62336 \ldots, \\
    \nonumber    B & = 0.23268 \ldots, \qquad \hspace{1.4mm} \beta = 1.37394 \ldots.
\end{align}
Considering the first 900 coefficients numerically yields 
$$ M := \{7, 26, 59, 104, 162, 233, 316, 412, 521, 642, 776, ...\},$$ 
which are the highest indices until a change of signs in the sequence $p(1,6,n) - p(5,6,n)$. We can compare this exact result to the prediction of formula \eqref{E:p6explicit}. By considering the roots of the cosine, we find that it changes signs approximately at 
$$M' := \{7, 27, 59, 104, 162, 233, 316, 412, 521, 642, 777, ...\}.$$
Note that in the first eleven cases, only case two with $27$ and case eleven with $777$ give slightly wrong predictions. 

\begin{figure}[H]
         \centering
         \includegraphics[width=120mm, height= 70mm]{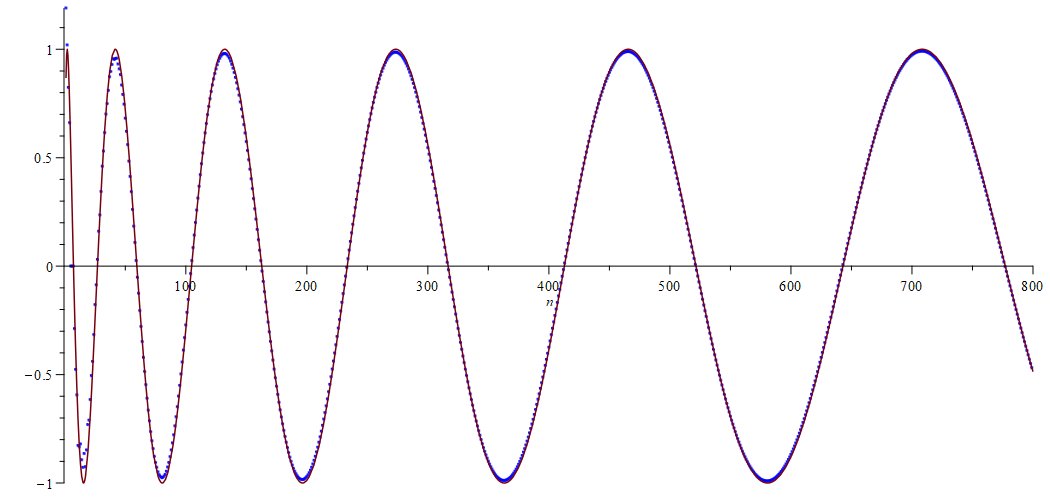}
        \caption{The plot shows the sign changes of $p(1,6,n)-p(5,6,n)$ (in blue dots) and the estimated sign changes (the red line).}
\label{fig:difffigintro}

\end{figure}

\end{example}
The next example refers to higher differences of partition functions. 
\begin{example} Again, we consider the case $b=6$. In the spirit of Theorem \ref{T:differenceswithsets} we want to consider multi-termed differences. To do so, we need two subsets $S_1, S_2 \subset \Z/6\Z$, such that the corresponding nontrivial polynomial $P_{S_1, S_2}(x)$ of degree at most 5 vanishes at as many roots of unity around $x=1$ as possible. We first note that
\begin{align*}
    (x-1)\Phi_6(x) = (x-1)\left(x^2 - x + 1\right) = x^3 - 2x^2 + 2x - 1.
\end{align*}
In this case, we obtain a weighted difference and no subsets can be found for a growth reduction to an exponent induced by $\mathrm{Li}_2(\zeta_3)$. We continue by trying to eliminate also the 3rd roots of unity; i.e., with 
\begin{align*}
    (x-1)\Phi_6(x)\Phi_3(x) = x^5 - x^4 + x^3 - x^2 + x - 1.
\end{align*}
This is exactly the case (4) of Theorem \ref{T:differenceswithsets}. As a result, setting $S_1 := \{1,3,5\}$ and $S_2 := \{0,2,4\}$, we find, as $n \to \infty$,
\begin{align*}
    \sum_{0 \leq a \leq 5} (-1)^{a+1} p(a,6,n) \sim - Q_n(-1) \sim \frac{(-1)^{n+1}}{2(24)^{\frac14} n^{\frac34}} \exp\left( \pi \sqrt{\frac{n}{6}}\right).
\end{align*}
Also note that $\frac{\pi}{\sqrt{6}} = 1.2825 \ldots$ is much smaller then the exponent $2\lambda_1 = 1.6281 \ldots$ in \eqref{E:p6explicit}, compare also \eqref{E:lambda1}. 
\end{example}
Finally, we give an application of Remark \ref{R:differenceremark}.

\begin{example} In light of Equation \eqref{E:phaseshift}, we can choose $r$ such that the cosine becomes a sine on the right hand side. Using Pythagoras' Theorem, and considering the case $b=8$ and $r=2$, we obtain an asymptotic formula without oscillating terms. Indeed, according to Theorem \ref{T:simpledifferences} (4) and Remark \ref{R:differenceremark}, respectively, we obtain for residue classes $a_1 \not= a_2$ that
\begin{align*}
    (p(a_1, 8, n) - p(a_2, 8, n))^2 + (p(a_1+2,8,n) - p(a_2+2,8,n))^2 \sim B_{a_1,a_2,8}^2 n^{-\frac32} \exp\left( 4\lambda_1 \sqrt{n}\right).
\end{align*}
There is no difficulty to extend this type of asymptotic formula for higher differences in the spirit of Theorem \ref{T:differenceswithsets} (1).
\end{example}

\section{Proof of Theorem \ref{T:twistedetaproduct}}\label{proofofmainthm} 
Since the asymptotic formulas for $Q_n(1)$ and $Q_n(-1)$ are well-known, we assume throughout that $1 \leq a < \frac{b}{2}$ with $\gcd(a,b)=1$ and $b \geq 3,$ since $\overline{Q_n(\zeta)} = Q_n(\overline{\zeta})$.   

The setup follows the standard Hardy--Ramanujan circle method, expressing $Q_n(\zeta)$ as a contour integral about 0 and breaking the contour apart with the sequence of Farey fractions of order $N.$  For facts about the Farey sequence, we refer the reader to Chapter 3 of \cite{HardyWright}.  Much of the analysis in this section closely follows \cite{Parry} after assuming the technical Lemmas \ref{L:Esmallkbound} and \ref{L:Elargekbound} which we prove in the next section.

Let $N=\lfloor \delta \sqrt{n} \rfloor$, for some $\delta>0$ to be chosen independently of $n$ and small enough during the course of the proof.  Let $\mathcal{F}_N$ be the sequence of Farey fractions of order $N$ with mediants $\theta_{h,k}'$ and $\theta_{h,k}''$ at $\frac{h}{k}$.  We write \begin{equation}\label{E:tthetadef} t_{\theta}:=t_n- 2 \pi i \theta, \qquad \text{where} \qquad t_n\sqrt{n}:=\frac{\sqrt{\mathrm{Li}_2\left(\zeta_b^{ak_0}\right)}}{k_0},\end{equation} with $k_0 \in \{1,2,3\}$ according to whether we are in case (1), (2) or (3).

  By Cauchy's integral formula, we have
\begin{align*}
    Q_n\left(\zeta_b^a\right)&=\int_{0}^{1} \left(\zeta_b^ae^{-t_n+2\pi i \theta};e^{-t_n+2\pi i \theta}\right)_{\infty}^{-1}e^{nt_n-2 \pi i n \theta} d \theta \\
    &=\sum_{\frac{h}{k} \in \mathcal{F}_N } \zeta_k^{-hn}  \int_{-\theta_{h,k}'}^{\theta_{h,k}''} \exp\left(\frac{\mathrm{Li}_2\left(\zeta_b^{ak}\right)}{k^2t_{\theta}} +nt_{\theta} + E_{h,k}(\zeta_b^a,t_{\theta})\right)d\theta,
\end{align*}
where
\begin{align*}E_{h,k}(z,t):=& \ \Log\left( \left(z \zeta_k^h e^{-t};\zeta_k^h  e^{-t}\right)_{\infty}^{-1}\right)-\frac{\mathrm{Li}_2\left(z^k\right)}{k^2t}.\end{align*}

We will show that the integral(s) where $k=k_0$ dominate, that is they are the {\it major arcs}, and all the other integrals are exponentially smaller, that is they are {\it minor arcs}.  The $E_{h,k}$ will be shown to be error terms on all arcs; the following gives the growth of the $E_{h,k}$ up to $o(1)$ on each of the possible major arcs.

\begin{lemma}\label{L:Efixedkbound}
\textup{(1)} For $2\pi \frac{a}{b} \in (0, \theta_{13})$ and $-\theta_{0,1}' \leq \theta \leq \theta_{0,1}''$, we have
$$
E_{0,1}(\zeta_b^a,t_{\theta})= \Log\left(\omega_{0,1}\left(\zeta_b^a\right)\right)+o(1).
$$
\textup{(2)} For $2\pi \frac{a}{b} \in (\theta_{23}, \pi)$ and $-\theta_{1,2}' \leq \theta \leq \theta_{1,2}''$, we have
$$
E_{1,2}(\zeta_b^a,t_{\theta})= \Log\left(\omega_{1,2}\left(\zeta_b^a\right)\right)+o(1).
$$
\textup{(3)} For $2\pi \frac{a}{b} \in (\theta_{13},\theta_{23})\setminus \left\{\frac{2\pi}{3}\right\}$ and $-\theta_{1,3}' \leq \theta \leq \theta_{1,3}''$, we have
$$
E_{1,3}(\zeta_b^a,t_{\theta})= \Log \left(\omega_{1,3}(\zeta_b^a)\right) + o(1),
$$
and for  $-\theta_{2,3}' \leq \theta \leq \theta_{2,3}''$, we have
$$
 E_{2,3}(\zeta_b^a,t_{\theta})= \Log \left(\omega_{2,3}(\zeta_b^a)\right) + o(1).
$$
\textup{(4)} For $\frac{a}{b}=\frac{1}{3}$ and $-\theta_{1,3}' \leq \theta \leq \theta_{1,3}''$, we have
\begin{align*}
& E_{1,3}(\zeta_3,t_{\theta})=\Log\left(t_{\theta}^{\frac16}\right)+\Log\left( \left(1-\zeta_3^2\right)^{-\frac16}\left(1-\zeta_3\right)^{\frac12} \right)+\log\left(\Gamma\left(\frac{2}{3}\right)\right)\\ & \qquad \hspace{3cm} +\frac{1}{6}\log\left(3\right) -\frac{1}{2}\log(2\pi)+ o(1),
\end{align*} 
and for $-\theta_{2,3}' \leq \theta \leq \theta_{2,3}''$, we have 
\begin{align*}
& E_{2,3}(\zeta_3,t_{\theta})=\Log\left(t_{\theta}^{-\frac16}\right)+\Log\left( \left(1-\zeta_3^2\right)^{\frac16}(1-\zeta_3)^{\frac12} \right)+\log\left(\Gamma\left(\frac{1}{3}\right)\right)\\ & \qquad \hspace{3cm} +\frac{1}{6}\log\left(\frac{1}{3}\right) -\frac{1}{2}\log(2\pi)+ o(1).
\end{align*}
\end{lemma}

The next two lemmas give uniform bounds on $E_{h,k}$ to be applied on the minor arcs.  The proofs are quite intricate and are provided in the next section. We assume throughout that $0<\varepsilon\leq\frac{1}{4}$.

\begin{lemma}\label{L:Esmallkbound}
Uniformly for $k \leq n^{\varepsilon}$ and $-\theta_{h,k}' \leq \theta \leq \theta_{h,k}''$, we have
$$
E_{h,k}(\zeta_b^a,t_{\theta})=O\left(n^{3\varepsilon-\frac{1}{2}}\right)+O\left(n^{\varepsilon}\right).
$$
\end{lemma}
The next lemma treats the case of large denominators. 
\begin{lemma}\label{L:Elargekbound}
Uniformly for $k \geq n^{\varepsilon}$ and $-\theta_{h,k}' \leq \theta \leq \theta_{h,k}''$, we have
$$
E_{h,k}(\zeta_b^a,t_{\theta})=O\left(N\right).
$$
\end{lemma}

We postpone the proofs of Lemmas \ref{L:Efixedkbound}--\ref{L:Elargekbound} until Section 5.  With these key lemmas in hand, the proof of Theorem \ref{T:twistedetaproduct} follows \cite{Parry} closely.

\begin{proof}[Proof of Theorem \ref{T:twistedetaproduct}]
{\it Cases (1), (2) and (3).} Assume $\frac{a}{b} \neq \frac{1}{3}.$
We write
$$
\lambda_1+i\lambda_2:=\frac{\sqrt{\mathrm{Li}_2(\zeta_b^{ak_0})}}{k_0}.
$$
It follows from the choice $0 < \varepsilon \leq \frac14$, Lemmas \ref{L:Esmallkbound} and \ref{L:Elargekbound} that $E_{h,k}(\zeta_b^a,t_{\theta})=\delta O(\sqrt{n})$ uniformly.  Using this and Lemma \ref{L:Efixedkbound}, we have
\begin{align}
    &e^{-2\lambda_1\sqrt{n}}Q_n\left(\zeta_b^a \right) \nonumber \\
    &=\sum_{\substack{h \\ \frac{h}{k_0} \in \mathcal{F}_N}} \zeta_{k_0}^{-hn}\omega_{h,k_0}(\zeta_b^a)\int_{-\theta_{h,k_0}'}^{\theta_{h,k_0}''} \exp\left(-2\lambda_1\sqrt{n}+\frac{(\lambda_1+i\lambda_2)^2}{t_{\theta}}+nt_{\theta} + o(1)\right)d\theta \nonumber  \\
    & \ \ \ +\sum_{\substack{\frac{h}{k} \in \mathcal{F}_N \\ k \neq k_0}} \zeta_{k}^{-hn}  \int_{-\theta_{h,k}'}^{\theta_{h,k}''} \exp\left(-2\lambda_1\sqrt{n}+\frac{\mathrm{Li}_2\left(\zeta_b^{ak}\right)}{k^2t_{\theta}} +nt_{\theta} + \delta O\left(\sqrt{n}\right)\right)d\theta. \label{E:Q_nmajorandminorarcs}
\end{align}
Recalling \eqref{E:tthetadef}, we rewrite the first term in \eqref{E:Q_nmajorandminorarcs} as
\begin{align*}
    &\sum_{\substack{1 \leq h < k_0 \\ (h,k)=1}}\zeta_{k_0}^{-hn}\omega_{h,k_0}\left(\zeta_b^a \right)\exp\left(i2\lambda_2\sqrt{n}\right) \int_{-\theta_{h,k_0}'}^{\theta_{h,k_0}''} \exp\left( \sqrt{n}\left(\frac{\lambda_1+i\lambda_2}{1-\frac{2\pi i \theta \sqrt{n}}{\lambda_1+i\lambda_2}}-(\lambda_1+i\lambda_2) \right) - 2 \pi in\theta \right) d \theta.
\end{align*}
We can estimate the mediants as $\frac{1}{\sqrt{n}}\ll \theta_{h,k_0}', \theta_{h,k_0}'' \ll \frac{1}{\sqrt{n}}$ and setting $\theta \mapsto \theta n^{-\frac12}$, the above integral is asymptotic to
\begin{equation}\label{E:majorarclaplaceint}
\frac{1}{n^{\frac12}}\int_{-c}^{c}
\exp\left( \sqrt{n}\left(\frac{\lambda_1+i\lambda_2}{1-\frac{2\pi i  \theta}{\lambda_1+i\lambda_2}}-(\lambda_1+i\lambda_2) -2 \pi i \theta\right)  \right) d \theta=:\frac{1}{n^{\frac12}}\int_{-c}^ce^{\sqrt{n}B(\theta)}d\theta,\end{equation}   
for some $c>0$, say.  We claim that we can apply Theorem \ref{laplacemethod} with $x_0 \mapsto 0,$ $A(x) \mapsto 1$ and $B$ as above.  Here, $B(0)=0$ and expanding the geometric series gives
$$
B(\theta)=-\frac{4 \pi^2}{\lambda_1+i\lambda_2}\theta^2+o\left(\theta^2\right), \qquad \theta \to 0,
$$
with $\mathrm{Re}\left(\frac{4 \pi^2}{\lambda_1+i\lambda_2} \right) > 0$.  Finally, we claim that $\mathrm{Re}(B(\theta)) \leq 0$ with equality if and only if $\theta=0.$   Indeed, 
$$
\mathrm{Re}(B(\theta))=\mathrm{Re}\left(\frac{\lambda_1+i\lambda_2}{1-\frac{2\pi i \theta}{\lambda_1+i\lambda_2}}\right)-\lambda_1 
=\frac{|\lambda_1+i\lambda_2|^2}{\lambda_1}\mathrm{Re}\left(\frac{e^{i\psi}}{1+i\left(\frac{\lambda_2}{\lambda_1}-\frac{2\pi \theta}{\lambda_1}\right)} \right)-\lambda_1,
$$
where $\lambda_1+i\lambda_2=\left|\lambda_1+i\lambda_2 \right|e^{i\frac{\psi}{2}}$ and $\psi \in \left(-\frac{\pi}{2}, \frac{\pi}{2}\right).$  Lemma \ref{L:exponentialquotientmax} applies to show
$$
\mathrm{Re}(B(\theta)) \leq \frac{|\lambda_1+i\lambda_2|^2}{\lambda_1}\cos^2\left(\frac{\psi}{2}\right)-\lambda_1=\frac{\lambda_1^2}{\lambda_1}-\lambda_1=0,
$$
with equality if and only if
$$
\Arg\left(1+i\left(\frac{\lambda_2}{\lambda_1}-\frac{2\pi\theta}{\lambda_1}\right)\right)=\frac{\psi}{2}=\Arg\left(\lambda_1+i\lambda_2\right),
$$
thus if and only if $\theta=0,$ as claimed.  Now by Theorem \ref{laplacemethod}, we conclude that \eqref{E:majorarclaplaceint} is asymptotic to
$$
\frac{\sqrt{\lambda_1+i\lambda_2}}{2\sqrt{\pi}n^{\frac34}},
$$
and overall
\begin{align*}
    &\sum_{\substack{1 \leq h < k_0 \\ (h,k)=1}}\zeta_{k_0}^{-hn}\omega_{h,k_0}\left(\zeta_b^a\right) \times \int_{-\theta_{h,k_0}'}^{\theta_{h,k_0}''} \exp\left( -2\lambda_1\sqrt{n} +\frac{(\lambda_1+i\lambda_2)^2}{\frac{\lambda_1+i\lambda_2}{\sqrt{n}}-2 \pi i \theta} +\sqrt{n}(\lambda_1+i\lambda_2)-2 \pi i n\theta \right) d \theta \\
    &\hspace{4cm} \sim \frac{\sqrt{\lambda_1+i\lambda_2}}{2\sqrt{\pi}n^{\frac34}}e^{2i\lambda_2\sqrt{n}}\sum_{\substack{1 \leq h < k_0 \\ (h,k_0)=1}} \zeta_{k_0}^{-hn}\omega_{h,k_0}\left(\zeta_b^a\right).
\end{align*}
When the $e^{2\lambda_1\sqrt{n}}$ is brought back to the right-hand side, this is the right-hand side of Theorem \ref{T:twistedetaproduct}.

For $k\neq k_0$, we follow Parry in Lemma 5.2 of \cite{Parry} and write
\begin{align*}
    \mathrm{Re}\left(\frac{\mathrm{Li}_2\left(\zeta_b^{ak}\right)}{k^2t_{\theta}} + nt_{\theta}\right) &=\lambda_1\sqrt{n}\left(\frac{\left|\mathrm{Li}_2\left(\zeta_b^{ak}\right)\right|}{k^2\lambda_1^2} \mathrm{Re}\left(\frac{e^{i\psi_k}}{1+i\left(\frac{\lambda_2}{\lambda_1}-\frac{2\pi\theta\sqrt{n}}{\lambda_1}\right)}\right)+1\right),
\end{align*}
where $\psi_k$ satisfies $\mathrm{Li}_2(\zeta_b^{ak})=|\mathrm{Li}_2(\zeta_b^{ak})|e^{i\psi_k}.$  Arguing as for the major arcs, the expression $\mathrm{Re}(\cdot)$ above is at most $\cos^2\left(\frac{\psi_k}{2}\right)$.   Now let
$$
\Delta:= \inf_{k \neq k_0}\left(1-\left(\mathrm{Re}\left(\frac{\sqrt{\mathrm{Li}_2(\zeta_b^{ak})}}{k\lambda_1}\right)\right)^2\right) >0.
$$
Then
$$
    \mathrm{Re}\left(\frac{\mathrm{Li}_2(\zeta_b^{ak})}{k^2t_{\theta}} + nt_{\theta}\right) \leq \lambda_1\sqrt{n}\left(\frac{\left|\mathrm{Li}_2\left(\zeta_b^{ak}\right)\right|}{k^2\lambda_1^2}\cos^2\left(\frac{\psi_k}{2}\right)+1\right)\leq \lambda_1 \sqrt{n}(2-\Delta).
$$
Thus,
\begin{align*}
    &\left|\sum_{\substack{\frac{h}{k} \in \mathcal{F}_N \\ k \neq k_0}} \zeta_{k}^{-hn}  \int_{-\theta_{h,k}'}^{\theta_{h,k}''}\exp\left(-2\lambda_1\sqrt{n}+\frac{\mathrm{Li}_2(z^k)}{k^2t_{\theta}} +nt_{\theta} + \delta O(\sqrt{n})\right)d\theta\right| \leq \exp\left(-\lambda_1\Delta \sqrt{n}+ \delta O\left(\sqrt{n}\right) \right).
\end{align*}
We can choose $\delta$ small enough so that the constant in the exponential is negative, and the minor arcs are exponentially smaller than the major arc(s).  This completes the proof of cases (1), (2) and (3).

For case (4), we have $\lambda_1+i\lambda_2=\lambda_1=\frac{\sqrt{\mathrm{Li}_2(1)}}{3}=\frac{\pi}{3\sqrt{6}}$.  Exactly as in case (3) the minor arcs are those with $k \neq 3$, and these are shown to be exponentially smaller than for $k_0=3$.  Thus, by Lemma \ref{L:Efixedkbound} part (4), we have
\begin{align*}
    e^{-2\lambda_1\sqrt{n}}Q_n(\zeta_3) &= \zeta_3^{-n}C_{1,3} \int_{-\theta_{1,3}'}^{\theta_{1,3}''} t_{\theta}^{\frac16} \exp\left(-2\lambda_1\sqrt{n}+\frac{\lambda_1^2}{t_{\theta}}+nt_{\theta}+o(1) \right)d \theta \nonumber \\
    & \quad + \zeta^{-2n}_3C_{2,3} \int_{-\theta_{2,3}'}^{\theta_{2,3}''} t_{\theta}^{-\frac16} \exp\left(-2\lambda_1\sqrt{n}+\frac{\lambda_1^2}{t_{\theta}}+nt_{\theta}+o(1) \right)d \theta, 
\end{align*}
where
\begin{align*}
    C_{1,3}&:=\left(1-\zeta_3^2\right)^{-\frac{1}{6}}(1-\zeta_3)^{\frac{1}{2}}\Gamma\left(\frac{2}{3}\right) \frac{3^{\frac{1}{6}}}{\sqrt{2\pi}}, \hspace{1cm} C_{2,3}&:=\left(1-\zeta_3^2\right)^{\frac{1}{6}}(1-\zeta_3)^{\frac{1}{2}}\Gamma\left(\frac{1}{3}\right) \frac{1}{3^{\frac{1}{6}}\sqrt{2\pi}}.
\end{align*}
Noting
$$
t_{\theta}=\frac{\pi}{3\sqrt{6n}}\left(1-6i\sqrt{6n}\theta\right),
$$
we have
\begin{align}
 \nonumber
    \begin{split}
    &e^{-2\lambda_1\sqrt{n}}Q_n(\zeta_3)
    \end{split}\\
      \nonumber
    \begin{split}
     &=\zeta_3^{-n}C_{1,3}\frac{(\pi)^{\frac16}}{3^{\frac16}(6n)^{\frac{1}{12}}} \int_{-\theta_{1,3}'}^{\theta_{1,3}''} \left(1-6i\sqrt{6n}\theta\right)^{\frac16} \exp\left(-2\lambda_1\sqrt{n}+\frac{\lambda_1^2}{t_{\theta}}+nt_{\theta}+o(1) \right)d \theta 
    \end{split}\\ 
    \begin{split}\label{e13majorarcsint2b} 
& + \zeta^{-2n}_3C_{2,3}\frac{3^{\frac16}(6n)^{\frac{1}{12}}}{(\pi)^{\frac16}} \int_{-\theta_{2,3}'}^{\theta_{2,3}''} \left(1-6i\sqrt{6n}\theta\right)^{-\frac16} \exp\left(-2\lambda_1\sqrt{n}+\frac{\lambda_1^2}{t_{\theta}}+nt_{\theta}+o(1) \right)d \theta.
    \end{split}
\end{align}
Setting $\theta \mapsto \theta n^{-\frac12}$ and arguing as before, both integrals are asymptotic to
$$
\frac{\sqrt{\lambda_1}}{2\sqrt{\pi}n^{\frac34}}=\frac{1}{2^{\frac54}3^{\frac34}n^{\frac34}}.
$$
Hence, the second term in \eqref{e13majorarcsint2b} dominates, and gives the claimed asymptotic formula.
\end{proof}

  \section{Proof of Lemmas \ref{L:Efixedkbound}, \ref{L:Esmallkbound}, and \ref{L:Elargekbound} }\label{minorarcbounds}

We prove Lemma \ref{L:Esmallkbound} first then make use of these ideas in the proof of Lemma \ref{L:Efixedkbound}. We finish the section by proving Lemma \ref{L:Elargekbound}.
\subsection{Proof of Lemma \ref{L:Esmallkbound} } We rewrite $E_{h,k}$ as a sum of two functions: a function to which we can apply Euler--Maclaurin summation, and another to which we apply the tools in Proposition \ref{P:Abelpartialsummation} and Lemma \ref{L:CosSumBound}.
\begin{lemma}\label{L:Ehkcase1rewrite}
For $\mathrm{Re}(t)>0$ and $z=\zeta_b^a,$ we have
\begin{align}
\label{E:Ehkcase2rewrite1} E_{h,k}(z,t)&=\sum_{\substack{1 \leq m \leq bk \\ 1 \leq j \leq k}}  \zeta_b^{ma} \zeta_k^{jmh}kt\sum_{\ell \geq 0} g_{j,k}\left(t(bk^2 \ell + km) \right)
 + \Log \left(\prod_{j=1}^k \left(1-\zeta_b^{a} \zeta_k^{-jh}e^{-jt}\right)^{-\frac{1}{2}+\frac{j}{k}}\right),
\end{align}
  where
  $$
  g_{j,k}(w):=\frac{e^{-\frac{j}{k}w}}{w(1-e^{-w})}-\frac{1}{w^2}-\left(\frac{1}{2}-\frac{j}{k} \right)\frac{e^{-\frac{j}{k}w}}{w}.
  $$
\end{lemma}
\begin{proof}
In $E_{h,k}$, we expand the logarithm using its Taylor series as
  \begin{align*}
    \Log\left(\zeta_b^{a} \zeta_k^{h}e^{-t}; \zeta_k^{h}e^{-t}\right)_{\infty}^{-1}
    &=\sum_{\substack{\nu \geq 1 \\ \ell \geq 1}} \frac{ \zeta_b^{\ell a} \zeta_k^{\ell \nu h}e^{-\ell \nu t}}{\ell} =\sum_{\substack{1 \leq j \leq k \\ 1 \leq m \leq bk}}\sum_{\substack{\nu \geq 0 \\ \ell \geq 0}} \frac{\zeta_{bk}^{m(ka+bjh)}e^{-(bk\ell+m) (\nu k + j)t}}{bk\ell+m} \\
     &=\sum_{\substack{1 \leq j \leq k \\ 1 \leq m \leq bk}}\zeta_{b}^{ma}\zeta_{k}^{mjh}\sum_{\ell \geq 0} \frac{e^{-jt(bk\ell+m)}}{(bk\ell+m)(1-e^{-t(bk^2\ell+km)})}.
  \end{align*}
  This corresponds to the left term in $g_{j,k}.$
  For the middle term in $g_{j,k}$, we compute (using $\gcd(h,k)=1$)
 $$
      \sum_{\substack{1 \leq j \leq k \\ 1 \leq m \leq bk}}\zeta_{b}^{ma}\zeta_{k}^{mjh}kt\sum_{\ell \geq 0} \frac{1}{t^2(bk^2\ell +km)^2}  =\frac{k^2}{t}\sum_{\substack{1 \leq m \leq bk \\ m \equiv 0 \pmod{k}}}\sum_{\ell \geq 0} \zeta_{b}^{ma}\frac{1}{(bk^2\ell+mk)^2}  =\frac{1}{tk^2} \mathrm{Li}_2\left(\zeta_{b}^{ka} \right).
$$
  
  Finally, it is simple to show that the logarithm of the product in \eqref{E:Ehkcase2rewrite1} cancels with the sum of the right term in $g_{j,k}$, simply by expanding the logarithm into its Taylor series.
\end{proof}

We estimate the first term in \eqref{E:Ehkcase2rewrite1} using Euler--Maclaurin summation.  First we need a technical definition.  Since $\textnormal{gcd}(h,k) =1$, there is at most one $j_0$ in the sum in \eqref{E:Ehkcase2rewrite1} for which $\zeta_{b}^{a}\zeta_{k}^{j_0h}=1;$ i.e., such that $ak+bj_0h\equiv 0 \pmod{bk}.$ Define
$$
\mathcal{S}_{a,b}:= \{(h,k) \in \mathbb{N}^2 : \text{there exists $j_0 \in [1,k]$ with $ak + j_0b h \equiv 0 \pmod{bk}$}\}.
$$



\begin{lemma}\label{L:gjksmallbound} Let $j_0$ be as above.  For $k \leq n^{\varepsilon}$ and $-\theta_{h,k}' \leq \theta \leq \theta_{h,k}''$, we have
  \begin{align*}
  &\sum_{\substack{1 \leq m \leq bk \\ 1 \leq j \leq k}}  \zeta_b^{ma} \zeta_k^{jmh}kt_{\theta}\sum_{\ell \geq 0} g_{j,k}\left(t_{\theta}(bk^2 \ell + km) \right)\\&=\left(\log \left(\Gamma\left(\frac{j_0}{k}\right)\right)+\left(\frac{1}{2}-\frac{j_0}{k}\right)\log\left(\frac{j_0}{k}\right)-\frac{1}{2}\log(2 \pi)\right)1_{(h,k) \in \mathcal{S}_{a,b}} +O\left(\frac{k^3}{\sqrt{n}}\right)+O\left(\frac{k^5}{n}\right).
  \end{align*}
\end{lemma}

\begin{proof}
Note that the function $g_{j,k}(w)$ is holomorphic at 0 and in any cone $|\Arg(w)| \leq \frac{\pi}{2} - \eta.$  Also, $\theta_{h,k}', \theta_{h,k}'' \leq \frac{1}{kN}=O\left(\frac{1}{\sqrt{n}}\right)$ implies that $t_{\theta}$ lies in such a fixed cone (see also \cite{andrewsbook} on p. 75).  Thus, we can apply Theorem \ref{T:BMJS} to $g_{j,d}(z)$ with $w\mapsto t_{\theta}bk^2$, $a\mapsto \frac{m}{bk},$ and $N \mapsto 0$,
  \begin{align*}
  \sum_{\ell \geq 0} g_{j,k}\left(t_{\theta}bk^2\left(\ell+\frac{m}{bk}\right)\right)&=\frac{1}{t_{\theta}bk^2}\int_0^{\infty} g_{j,k}(w)dw-\left(\frac{1}{12}-\frac{j^2}{2k^2}\right)\left(\frac{1}{2}-\frac{m}{bk}\right)+O\left(\frac{k^2}{\sqrt{n}}\right).
  \end{align*}
  When summing the $O$-term, we get
  $$
  \sum_{\substack{1 \leq m \leq bk \\ 1 \leq j \leq k}} kt_{\theta} O\left(\frac{k^2}{\sqrt{n}}\right)=O\left(\frac{k^5}{n} \right),
  $$
  where we used the fact that $t_{\theta}$ lies in a cone. Summing first over $m$ gives
  $$
  \sum_{1 \leq m \leq bk} \zeta_{b}^{ma} \zeta_k^{jmh}\left(\frac{1}{12}-\frac{j^2}{2k^2}\right)\left(\frac{1}{2}-\frac{m}{bk}\right)= O\left(k\right).
  $$
  
  Hence,
  \begin{align*}
      &\sum_{\substack{1 \leq m \leq bk \\ 1 \leq j \leq k}}  \zeta_b^{ma} \zeta_k^{jmh}kt_{\theta}\sum_{\ell \geq 0} g_{j,k}\left(t_{\theta}(bk^2 \ell + km) \right)\\
      &=\left(\int_0^{\infty} g_{j_0,k}(w)dw\right)1_{(h,k) \in \mathcal{S}_{a,b}}+O\left(\frac{k^3}{\sqrt{n}}\right) +O\left(\frac{k^2}{\sqrt{n}}\right)+O\left(\frac{k^5}{n}\right) \\
      &=\left(\log \Gamma\left(\frac{j_0}{k}\right)+\left(\frac{1}{2}-\frac{j_0}{k}\right)\log\left(\frac{j_0}{k}\right)-\frac{1}{2}\log(2 \pi)\right)1_{(h,k) \in \mathcal{S}_{a,b}} +O\left(\frac{k^3}{\sqrt{n}}\right)+O\left(\frac{k^5}{n}\right),
  \end{align*}
where the last step follows by  by Lemma \ref{L:Bringmannintegral}. If $\varepsilon <\frac{1}{4},$ then the terms $O\left(\frac{k^3}{\sqrt{n}}\right)$ and $O\left(\frac{k^5}{n}\right)$ are smaller than $\sqrt{n}$ as needed since $N = O(\delta \sqrt{n})$ where $\delta$ is chosen small and independently of $n$.
\end{proof}
It remains to estimate the product term in \eqref{E:Ehkcase2rewrite}.
\begin{lemma}\label{L:TheProduct}
  Uniformly for $k \leq n^{\varepsilon}$ and $-\theta_{h,k} \leq \theta \leq \theta_{h,k}''$, we have
  $$
  \Log\left(\prod_{j=1}^k \left(1-\zeta_b^{a} \zeta_k^{-jh}e^{-jt_{\theta}}\right)^{-\frac{1}{2}+\frac{j}{k}}\right) = O\left(n^{\varepsilon}\right).
  $$
\end{lemma}

    \begin{proof}
    Recall the sums $G_m$ defined in Lemma \ref{L:CosSumBound}.  Using the Taylor expansion for the logarithm followed by Proposition  \ref{P:Abelpartialsummation}, we write
    \begin{align*}
    \left|\Log \left(\prod_{j=1}^k \left(1-\zeta_b^{a} \zeta_k^{-jh}e^{-jt_{\theta}}\right)^{-\frac{1}{2}+\frac{j}{k}}\right) \right|&=
     \left|\sum_{j=1}^k\left(\frac{1}{2}-\frac{j}{k}\right)\sum_{m \geq 1} \frac{ \zeta_{bk}^{m(ak+bjh)}}{m}e^{-jmt_{\theta}} \right|\\&=\left|\sum_{j=1}^k\left(\frac{1}{2}-\frac{j}{k}\right)(1-e^{-jt_{\theta}})\sum_{m \geq 1} G_m\left(\frac{ak+bhj}{bk}\right)e^{-mjt_{\theta}}\right|.
    \end{align*}
    It is elementary to show that at most one of $\{ak+bjh\}_{1 \leq j \leq k}$ is divisible by $bk$.  Suppose that this happens at $j_0$ (if it never happens, then the argument is similar).  Then $G_m\left(\frac{ak+bhj_0}{bk}\right)=H_m$, the $m$-th harmonic number.  Applying the formula $\sum_{m \geq 1} H_mx^m=\frac{-\log(1-x)}{1-x}$, the above is

    \begin{align*}
    &\ll \left|1-e^{-j_0t_{\theta}}\right|\sum_{m \geq 1} H_m e^{-mj_0\mathrm{Re}(t_{\theta})}+ \sum_{\substack{1 \leq j \leq k \\ j \neq j_0}} \left|1-e^{-jt_{\theta}}\right|\max_{m \geq 1} \left| G_m\left(\frac{a}{b}+\frac{hj}{k}\right) \right|\sum_{m \geq 1} e^{-mj\mathrm{Re}(t_{\theta})} \\
    &\ll \left|\log\left(1-e^{-j_0t_{\theta}}\right)\right|\frac{|1-e^{-j_0t_{\theta}}|}{1-e^{-j_0\mathrm{Re}(t_{\theta})}}+ \sum_{\substack{1 \leq j \leq k \\ j \neq j_0}} \frac{\left|1-e^{-jt_{\theta}}\right|}{1-e^{-j\mathrm{Re}(t_{\theta})}}\max_{m \geq 1} \left| G_m\left(\frac{a}{b}+\frac{hj}{k}\right) \right|.
    \end{align*}
    The fact that $t_{\theta}$ lies in a cone $|\Arg(t_{\theta})|\leq \frac{\pi}{2}-\eta$ with $j \leq k \leq n^{\varepsilon}<\sqrt{n}$ gives
    $$\frac{\left|1-e^{-jt_{\theta}}\right|}{1-e^{-j\mathrm{Re}(t_{\theta})}}= O \left( \frac{|t_{\theta}|}{\mathrm{Re}(t_{\theta})}\right) =O(1).$$
     Thus, using Lemma \ref{L:Gmaxbound}
    \begin{align*}
    \left|\Log \left(\prod_{j=1}^k \left(1-\zeta_b^{a} \zeta_k^{-jh}e^{-jt_{\theta}}\right)^{-\frac{1}{2}+\frac{j}{k}}\right) \right|&=O(\log( n))+O\left(\sum_{\substack{1 \leq j \leq k \\ j \neq j_0}} \max_{m \geq 1} \left|G_m\left(\frac{a}{b}+\frac{hj}{k}\right)\right|\right) \\ &=O(\log (n))+O(k)=O(n^{\varepsilon}),
    \end{align*}
    as claimed.
   \end{proof}
   
   Lemma \ref{L:Esmallkbound} now follows from Lemmas \ref{L:Ehkcase1rewrite}, \ref{L:gjksmallbound} and  \ref{L:TheProduct}  by recalling that $k\leq n^{\varepsilon}$ where $0 < \varepsilon \leq \frac{1}{4}$:
   \begin{align*}
       E_{h,k}(\zeta^a_b,t_\theta) &= \left(\log \left(\Gamma\left(\frac{j_0}{k}\right)\right)+\left(\frac{1}{2}-\frac{j_0}{k}\right)\log\left(\frac{j_0}{k}\right)-\frac{1}{2}\log(2 \pi)\right)1_{(h,k) \in \mathcal{S}_{a,b}}\\ &\hspace{5mm}+O\left(\frac{k^3}{\sqrt{n}}\right)+O\left(\frac{k^5}{n}\right)+ O\left(n^\varepsilon\right)\\
       &\ll \log(n) + n^{3\varepsilon-\frac{1}{2}}+n^\varepsilon = O\left(n^{3\varepsilon-\frac{1}{2}}\right)+ O\left(n^\varepsilon\right).
   \end{align*}
   
   \subsection{Proof of Lemma  \ref{L:Efixedkbound}} To prove Lemma \ref{L:Efixedkbound}, we need an elementary fact about the sets $\mathcal{S}_{a,b}$.
   \begin{lemma}\label{L:Sabknonempty}
     Let $1 \leq a < \frac{b}{2}$ with $\gcd(a,b)=1$ and $b \geq 3$.  Then $(1,1)\not \in \mathcal{S}_{a,b}$, $(h,2)\not\in \mathcal{S}_{a,b}$, and $(h,3)\in \mathcal{S}_{a,b}$ if and only if $(a,b) =(1,3)$.
   \end{lemma}
   \begin{proof}
   We prove the case $(h,k) = (h,2)$ and note that the remaining cases are analogous. We have that $ (h,2) \in S_{a,b}$   if and only if $2b$ divides $2a+b$ or $2a+2b$.  Clearly, $2b \nmid (2a+2b)$, and since $2a+b < 3b < 2 \cdot (2b)$,
   $$
   2b \mid (2a+b) \iff 2a+b=2b \iff 2a=b \iff (a,b)=(1,2),
   $$
   which is a contradiction.
   \end{proof}

   \begin{proof}[Proof of Lemma \ref{L:Efixedkbound}]
   Cases (1), (2) and (3) are simple consequences of Lemmas \ref{L:Ehkcase1rewrite} and \ref{L:gjksmallbound} and \ref{L:Sabknonempty}.  For case (4), we suppose $\zeta= \zeta_3$.  Then one finds $j_0(1,3,1,3)=2$ and Lemmas \ref{L:Ehkcase1rewrite} and  \ref{L:gjksmallbound} imply 
   \begin{align*}
 E_{1,3}( \zeta_3,t_{\theta}) & =\Log\left(\prod_{j=1}^3\left(1- \zeta_3 \zeta_3^{j}e^{-jt_{\theta}}\right)^{-\frac{1}{2}+\frac{j}{3}}\right)+\log \left(\Gamma\left(\frac{2}{3}\right)\right)-\frac{1}{6}\log\left(\frac{2}{3}\right)-\frac{1}{2}\log(2\pi)+o(1)\\
 &=\Log\left(t_{\theta}^{\frac{1}{6}}\right)+\Log\left(\frac{(1- \zeta_3)^{\frac12}}{(1- \zeta_3^2)^{\frac16}} \right)+\log\left( \Gamma\left(\frac{1}{3}\right)\right)+\frac{1}{6}\log\left(3\right)-\frac{1}{2}\log(2\pi)+o(1) 
 ,
   \end{align*}
   as claimed, whereas $j_0(1,3,2,3)=1$, and so
   \begin{align*}
       &E_{2,3}(\zeta_b^{a},t_{\theta}) =\log \prod_{j=1}^3\left(1- \zeta_3 \zeta_3^{2j}e^{-jt_{\theta}}\right)^{-\frac{1}{2}+\frac{j}{3}}+\log \Gamma\left(\frac{1}{3}\right) + \frac{1}{6}\log\left(\frac{1}{3}\right) - \frac{1}{2}\log(2 \pi)+o(1) \\
       &=\log\left(t_{\theta}^{-\frac16}\right)+\log\left( \left(1- \zeta_3^2\right)^{\frac16}(1- \zeta_3)^{\frac12} \right)+\log\Gamma\left(\frac{1}{3}\right)+\frac{1}{6}\log\left(\frac{1}{3}\right)  -\frac{1}{2}\log(2\pi)+ o(1),
   \end{align*}
   as claimed.
   \end{proof}

\subsection{Proof of Lemma \ref{L:Elargekbound}} In preparation for the proof of Lemma \ref{L:Elargekbound}, we rewrite $E_{h,k}$ as in Lemma \ref{L:Ehkcase1rewrite}, this time using only the first two terms of $g_{j,k}$.  The proof is analogous.
\begin{lemma}\label{L:Ehkcase2rewrite}
For $\mathrm{Re}(t)>0$ and $z=\zeta_b^{a},$ we have
\begin{equation}\label{E:Ehkcase2rewrite}
E_{h,k}(z,t)=\sum_{\substack{1 \leq m \leq bk \\ 1 \leq j \leq k}}  \zeta_b^{ma} \zeta_k^{jmh}kt\sum_{\ell \geq 0} \widetilde{g}_{j,k}\left(t\left(bk^2 \ell + km\right) \right),\end{equation}
  where
  $$
  \widetilde{g}_{j,k}(w):=\frac{e^{-\frac{j}{k}w}}{w(1-e^{-w})}-\frac{1}{w^2}.
  $$
\end{lemma}
We will need to estimate the sum in \eqref{E:Ehkcase2rewrite} separately for $\ell \geq 1$ and $\ell=0.$  For $\ell \geq 1$, we can first compute the sum on $j$ as
  $$
  \sum_{1\leq j \leq k} \zeta_k^{jmh}\widetilde{g}_{j,k}(z)=\frac{ \zeta_{k}^{mh}e^{-\frac{z}{k}}}{z(1- \zeta_{k}^{mh}e^{-\frac{z}{k}})}-\frac{k}{z^2} \cdot 1_{k \mid m}.
  $$
  Thus, writing $m=\nu k$ with $1 \leq \nu \leq b$ when $k \mid m$, we have
  
  \begin{align*}
     &\sum_{\substack{1 \leq m \leq bk \\ 1 \leq j \leq k}}  \zeta_b^{ma} \zeta_k^{jmh}kt\sum_{\ell \geq 1} \widetilde{g}_{j,k}\left(t(bk^2 \ell + km) \right) \\
     &=t\sum_{\nu=1}^b \zeta_{b}^{\nu ka}\sum_{\ell \geq 1} f_1\left(bkt\left(\ell + \frac{\nu}{b}\right)\right)+t\sum_{\substack{1 \leq m \leq bk \\ k \nmid m}} \zeta_{b}^{ma}\sum_{\ell \geq 1} f_2\left(bkt\left(\ell + \frac{m}{bk}\right)\right) =:S_1+S_2,
  \end{align*}
 say, where
  $$
  f_1(z):=\frac{e^{-z}}{z(1-e^{-z})}-\frac{1}{z^2}
  $$
  and
  $$
  f_2(z):=\frac{ \zeta_{k}^{mh}e^{-z}}{z(1- \zeta_{k}^{mh}e^{-z})}.
  $$
 \begin{lemma}\label{L:S1bound}
   For $k \geq n^{\varepsilon}$, $-\theta_{h,k}' \leq \theta \leq \theta_{h,k}''$ and $t \mapsto t_{\theta}$, we have $|S_1|=O\left(\log(n)\right)$.
 \end{lemma}
 \begin{proof}

  We use Theorem \ref{T:EulerMac} (with $N \to \infty$) to write
  \begin{align*}
  \sum_{\ell \geq 1} f_1\left(bkt_{\theta}\left(\ell + \frac{\nu}{b}\right)\right)&=\frac{f_1\left(bkt_{\theta}\left(1+\frac{\nu}{b}\right)\right)}{2}+\int_1^{\infty}f_1\left(bkt_{\theta}\left(x+\frac{\nu}{b}\right)\right)dx \\& \quad +bkt_{\theta}\int_1^{\infty}f_1'\left(bkt_{\theta}\left(x+\frac{\nu}{b}\right)\right)\left(\{x\}-\frac{1}{2}\right)dx \\
  &=\frac{f_1\left(kt_{\theta}\left(b+\nu\right)\right)}{2}+\frac{1}{bkt_{\theta}}\int_{kt_{\theta}(b+\nu)}^{\infty}f_1\left(z\right)dz \\& \quad +\int_{kt_{\theta}(b+\nu)}^{\infty}f_1'\left(z\right)\left(\left\{\frac{z-\frac{\nu}{b}}{bkt_{\theta}}\right\}-\frac{1}{2}\right)dz.
  \end{align*}
  Here, $f_1(z) \ll \frac{1}{z}$ as $z \to 0,$ thus
  $$
  f_1\left(kt_{\theta}(b+\nu)\right)=O\left(\frac{1}{k|t_{\theta}|} \right).
  $$
  Furthermore $\int_{t\cdot \frac{c}{|t|}}^{t\infty} f_1(z)dz=O(1)$ for $|\Arg(t)|\leq \frac{\pi}{2}-\eta,$ uniformly for any $\eta, c>0.$  As noted before, $t_{\theta}$ lies in such a cone, so
 $$
      \frac{1}{bkt_{\theta}}\int_{kt_{\theta}(b+\nu)}^{\infty}f_1\left(z\right)dz
      =O\left(\frac{1}{k|t_{\theta}|}\int_{kt_{\theta}(b+\nu)}^{t_\theta\frac{2b}{|t_\theta|}}\frac{1}{z}dz\right) =O\left(\frac{|\log(kt_{\theta})|}{k|t_{\theta}|}\right).
  $$
  Similarly, one has
  $$
  f_1'(z)=-\frac{e^{-z}}{z(1-e^{-z})}-\frac{e^{-z}}{z^2(1-e^{-z})}-\frac{e^{-2z}}{z(1-e^{-z})^2}+\frac{2}{z^3},
  $$
  so  $\int_{t\cdot \frac{c}{|t|}}^{t\infty} f'_1(z)dz=O(1)$ for $|\Arg(t)|\leq \frac{\pi}{2}-\eta,$ for any $\eta, c>0.$  And one has $f_1'(z) \ll \frac{1}{z^2}$ as $z \to 0,$ thus
 \begin{align*}
     \int_{kt_{\theta}(b+\nu)}^{\infty}f_1'\left(z\right)\left(\left\{\frac{z-\frac{\nu}{b}}{bkt_{\theta}}\right\}-\frac{1}{2}\right)dz &=O\left(\int_{kt_{\theta}(b+\nu)}^{t_{\theta}\frac{2b}{|t_{\theta}|}} \frac{1}{z^2} dz \right) =O\left(\frac{1}{k|t_{\theta}|}\right).
 \end{align*}
 The above bounds are all clearly uniform in $1 \leq \nu \leq b$, thus overall
 $$
 |S_1| = \sum_{\substack{1 \leq m \leq bk \\ k \nmid m}}|t_{\theta}|O\left(\frac{|\log kt_{\theta}|}{k|t_{\theta}|}\right)=O\left(|\log (kt_{\theta})|\right)=O(\log n),
 $$
as claimed.
  \end{proof}
\begin{lemma}\label{L:S2bound}
  For $k \geq n^{\varepsilon}$, $-\theta_{h,k}' \leq \theta \leq \theta_{h,k}''$ and $t \mapsto t_{\theta}$, we have $|S_2|=O\left(n^{\frac12-\varepsilon}\right)$.
\end{lemma}
\begin{proof}
For $\mathrm{Re}(z)>0$, we see immediately that
  $$
  |f_2(z)|\leq \frac{e^{-\mathrm{Re}(z)}}{\mathrm{Re}(z)(1-e^{-\mathrm{Re}(z)})}=:\tilde{f_2}(\mathrm{Re}(z)). 
  $$
  Thus, since $\tilde{f_2}$ is decreasing, we have 
  $$
  |S_2| \leq |t_{\theta}|bk \sum_{\ell \geq 1} \tilde{f_2}\left(bk\mathrm{Re}(t_{\theta})\ell\right) 
  \leq \frac{|t_{\theta}|}{\mathrm{Re}(t_{\theta})}\left(\int_{bk\mathrm{Re}(t_{\theta})}^{\infty} \tilde{f_2}(x)dx + bk\mathrm{Re}(t_{\theta}) \tilde{f_2}(bk\mathrm{Re}(t_{\theta}))  \right),
  $$
  by integral comparison.  We can bound the right term as
  $$
 bk\mathrm{Re}(t_{\theta}) \tilde{f_2}(bk\mathrm{Re}(t_{\theta}))=\frac{1}{e^{bk\mathrm{Re}(t_{\theta})}-1}\leq \frac{1}{bk\mathrm{Re}(t_{\theta})}.
  $$
  Furthermore, as $\eta \to 0^+$, we have
  $$
  \int_{\eta}^{\infty} \tilde{f_2}(x)dx = O(1)+\int_{\eta}^1 \frac{e^{-x}}{x(1-e^{-x})}dx \leq O(1)+\int_{\eta}^1 \frac{1}{x^2}dx=O\left(\frac{1}{\eta}\right).
  $$
  Hence, overall,
  $$
  |S_2| = O\left(\frac{1}{k\mathrm{Re}(t_{\theta})}\right)=O\left(\frac{\sqrt{n}}{k}\right)=O\left(n^{\frac12-\varepsilon} \right),
  $$
  as claimed.
\end{proof}

It remains to estimate the double sum \eqref{E:Ehkcase2rewrite} for the term $\ell=0;$ i.e., 
\begin{equation*}
    \sum_{\substack{1 \leq m \leq bk \\ 1 \leq j \leq k}} \zeta_{b}^{ma} \zeta_k^{jmh} \frac{\phi_{\frac{j}{k}}(tkm)}{m}, 
\end{equation*}
where
\begin{align*} 
\phi_a(w):=\frac{e^{-aw}}{1-e^{-w}}- \frac{1}{w}.
\end{align*}
Note that $\phi_a$ is holomorphic at 0 and in the cone $|\Arg(w)| \leq \frac{\pi}{2}- \eta,$ for any $\eta>0.$  We apply Lemma \ref{SmallRadiusEstimate} to $\phi_a$ to bound differences as follows.

\begin{lemma} \label{SmallRadiusEstimate} Let $x$ be a complex number with positive imaginary part and $|x| \leq 1$. Then there is a constant $c > 0$ independent from $a$ and $x$, such that for all $m \leq \frac{1}{|x|}$ we have
\begin{align*}
\left| \phi_a(xm) - \phi_a(x(m+1))\right| \leq c|x|.
\end{align*}
\end{lemma}
\begin{proof} The function $\phi_a(z)$ is holomorphic in $B_3(0)$. The functions $\phi_a'(z)$ are uniformly bounded on $\overline{B_{\frac52}(0)} \subset B_3(0)$. Indeed, we have uniformly in $a$
\begin{align*}
    & \max_{|z| \leq \frac52} |\phi'_a(z)| = \max_{|z| = \frac{5}{2}} |\phi'_a(z)| \leq \max_{|z| = \frac52} \left| \frac{e^{-az-z}}{(1-e^{-z})^2}\right| + \max_{|z| = \frac52} \left| \frac{ae^{-az}}{1-e^{-z}}\right| + \max_{|z| = \frac52} \left| \frac{1}{z^2}\right| \ll 1.
\end{align*} 
On the other hand, by Lemma \ref{L:DifferenceEstimate} applied to $f = \phi_a$, $U = B_3(0)$, and $\overline{B_{\frac52}(0)} \subset U$, we find, since $|xm| \leq 1$ and $|x(m+1)| \leq |mx| + |x| \leq 2$
\begin{align*}
|\phi_a(xm+x) - \phi_a(mx)| \leq \max_{|z| \leq \frac52} |\phi'_a(z)| |xm + x - xm| = c|x|,    \end{align*}
where $c$ does not depend on $0 < a \leq 1$. \end{proof}
We also require the following lemma for large values of $m$, whose proof is a straightforward calculation using that the denominators of the first term in $\phi_a(w)$ are bounded away from 0. 
\begin{lemma} \label{LargeRadiusEstimate} Let $x$ be a complex number with positive imaginary part. Then all $m > \frac{1}{|x|}$ we have
\begin{align*}
\left| \phi_a(xm) - \phi_a(x(m+1))\right| \ll \frac{1}{|x|m(m+1)} + |x|e^{-m\mathrm{Re}(x)} + a|x| e^{-am\mathrm{Re}(x)}.
\end{align*}
\end{lemma}

The following lemma, when combined with Lemmas \ref{L:Ehkcase2rewrite}--\ref{L:S2bound}, completes the proof of Lemma \ref{L:Elargekbound}, and thus that of Theorem \ref{T:twistedetaproduct}.
\begin{lemma}\label{L:l=0rewriteterm1}
For $ n^{\varepsilon} \leq k \leq N$ and $-\theta_{h,k}' \leq \theta \leq \theta_{h,k}''$, we have
$$
\sum_{\substack{1 \leq m \leq bk \\ 1 \leq j \leq k}}\zeta_b^{ma}\zeta_{k}^{jmh}\frac{\phi_{\frac{j}{k}}(t_{\theta}km)}{m}=O(k).
$$
\end{lemma}

\begin{proof} Let $x := kt_{\theta}$ and $a := \frac{j}{k}$. Note that we have $0 < a \leq 1$,  $\mathrm{Re}(x) > 0$, and $\frac{|x|}{\mathrm{Re}(x)} \ll 1$ uniformly in $k$.  We use Abel partial summation and split the sum into two parts:
    \begin{align*}
        \sum_{\substack{1 \leq m \leq bk \\ 1 \leq j \leq k}} = \sum_{j=1}^k \left( \sum_{0 < m \leq \min\left\{bk, \frac{1}{|x|}\right\}} + \sum_{\min\left\{bk, \frac{1}{|x|}\right\} < m \leq bk}\right).
    \end{align*}
    In the case $|x| > 1$, the first sum is empty, so we can assume $|x| \leq 1$. We first find with Proposition \ref{P:Abelpartialsummation} that
    \begin{align*}
         \sum_{m \leq \min\left\{bk, \frac{1}{|x|}\right\}}  \zeta_{bk}^{m(ak+hjb)}\frac{1}{m}\phi_a(xm) &= G_{\min\left\{bk, \frac{1}{|x|}\right\}}\left( \frac{a}{b} + \frac{hj}{k} \right) \phi_a\left( x \min\left\{bk, \Big\lfloor \frac{1}{|x|} \Big\rfloor \right\} \right) \\
        & + \sum_{m \leq \min\left\{bk, \frac{1}{|x|}\right\} - 1} G_{m}\left( \frac{a}{b} + \frac{hj}{k} \right) \left( \phi_a(mx) - \phi_a((m+1)x)\right).
    \end{align*}
    It follows that with Lemma \ref{SmallRadiusEstimate}
    \begin{align*}
    & \left| \sum_{j=1}^k \sum_{m \leq \min\left\{bk, \frac{1}{|x|}\right\}} \zeta_b^{ma}\zeta_k^{mhj}\cdot \frac{1}{m}\phi_a(mx) \right|  \leq \sum_{j=1}^k \left|G_{\min\left\{bk, \frac{1}{|x|}\right\}}\left( \frac{a}{b} + \frac{hj}{k} \right)\right| \left| \phi_a\left( x \min\left\{bk, \Big\lfloor\frac{1}{|x|}\Big\rfloor \right\}  \right) \right| \\
    & \hspace{3cm} + \left| \sum_{j=1}^k \sum_{0 < m \leq \min\left\{bk, \frac{1}{|x|}\right\}} G_{m}\left( \frac{a}{b} + \frac{hj}{k} \right) \left( \phi_a(mx) - \phi_a((m+1)x)\right) \right| \\
    & \ll \sum_{j=1}^k \left|G_{\min\left\{bk, \frac{1}{|x|}\right\}}\left( \frac{a}{b} + \frac{hj}{k} \right)\right| + \sum_{j=1}^k \max_{m=1,..., \min\left\{bk, \frac{1}{|x|}\right\}} \left|G_{m}\left( \frac{a}{b} + \frac{hj}{k} \right)\right| \sum_{0 < m \leq \frac{1}{|x|}} |x| = O(k),
    \end{align*}
    where we used Lemma \ref{L:Gmaxbound} and $G_{bk}(1) = H_{bk} = O(\log(k))$ in the last step. Similarly, we find with Lemma \ref{LargeRadiusEstimate} (without loss of generality we assume $\frac{1}{|x|} < bk$)
    \begin{align*}
        & \left| \sum_{j=1}^k \sum_{\frac{1}{|x|} < m \leq bk} \zeta_b^{ma}\zeta_{k}^{mhj}\cdot \frac{1}{m}\phi_a(mx) \right| \\
        & \ll \sum_{j=1}^k \left| G_{bk}\left( \frac{a}{b} + \frac{hj}{k}\right) \right| \left| \phi_a\left( x bk \right) \right| + \left| \sum_{j=1}^k \sum_{\frac{1}{|x|} < m \leq bk} G_{m}\left( \frac{a}{b} + \frac{hj}{k} \right) \left( \phi_a(mx) - \phi_a((m+1)x)\right) \right| \\
        & \ll O(k) + \sum_{j=1}^k \max_{m=\frac{1}{|x|},..., bk} \left|G_{m}\left( \frac{a}{b} + \frac{hj}{k} \right)\right| \sum_{\frac{1}{|x|} < m \leq bk} \left( \frac{1}{|x|m(m+1)} + |x| e^{-m \mathrm{Re}(x)} + a|x|e^{-am\mathrm{Re}(x)} \right). 
    \end{align*}
    Note that we uniformly have $\phi_a(xbk) \ll 1$ (as $1 \ll |xbk|$ and $x$ is part of a fixed cone \newline $|\Arg(x)|\leq \frac{\pi}{2}-\eta$) as well as 
    \begin{align*}
    \sum_{\frac{1}{|x|} < m \leq bk} \frac{1}{|x|m(m+1)} & \leq \sum_{\frac{1}{|x|} < m < \infty} \frac{1}{|x|m(m+1)} \ll 1
    \end{align*}
    and
    \begin{align*}
    \sum_{\frac{1}{|x|} < m \leq bk} |x| e^{-m \mathrm{Re}(x)} & \leq \frac{|x|}{1-e^{-\mathrm{Re}(x)}} \ll 1,
    \end{align*}
     as $\frac{|x|}{\mathrm{Re}(x)} \ll 1, |x| \ll 1.$ Similarly,
     
    \begin{align*}
    \sum_{\frac{1}{|x|} < m \leq bk} a|x| e^{-am \mathrm{Re}(x)} & \ll 1.
    \end{align*}
    As a result, using Lemma \ref{L:Gmaxbound} (again up to at most one summand in $O(\log(k))$),
        \begin{align*}
        \left| \sum_{j=1}^k \sum_{\frac{1}{|x|} < m \leq bk} \zeta_b^{ma}\zeta_{k}^{mhj}\cdot \frac{1}{m}\phi_a(mx) \right| & \ll O(k) + \sum_{j=1}^k \max_{m=\frac{1}{|x|},..., bk} \left|G_{m}\left( \frac{a}{b} + \frac{hj}{k} \right)\right| = O(k),
        \end{align*}
       as claimed. 
    \end{proof}
    
    \section*{Data Availability Statement} Data sharing not applicable to this article as no data sets
were generated or analysed during the current study.

\section*{Conflict of Interest Statement} There are no conflicts of interest for the current study.


\begin{thebibliography}{99}
	
	\bibitem{andrewsbook} G. Andrews.  {\it The Theory of Partitions}.  Cambridge University Press, 1984.
	
	\bibitem{BeckwithMertens1}  O. Beckwith and M. Mertens.  ``On the number of parts in integer partitions lying in given residue classes''.  {\it Annals of Combinatorics} {\bf 21} (2017).
	
	\bibitem{BeckwithMertens2}  O. Beckwith and M. Mertens.  ``The number of parts in certain residue classes of integer partitions''.  {\it Res. Number Theory} {\bf A11} (2015).
	
	\bibitem{Boyergoh}  R. Boyer and W. Goh.  ``Partition Polynomials: Asymptotics and Zeros''.  {\it Preprint} (2007). \texttt{https://arxiv.org/abs/0711.1373}
	
	\bibitem{Boyergoh2}  R. Boyer and W. Goh.  ``Polynomials associated with partitions: their asymptotics and zeroes''.  {\it Contemp. Math.} {\bf 471} (2007).
	
	\bibitem{Boyerkeith}  R. Boyer and W. Keith. ``Stabilization of coefficients for partitions polynomials''.  {\it Integers} {\bf A56} (2013).
	
	\bibitem{Boyerparry}  R. Boyer and D. Parry. ``Zero attractors of partition polynomials''.  {\it preprint} (2021). \texttt{https://arxiv.org/abs/2111.12226}
	
	\bibitem{BrigmannOnoMales}  K. Bringmann, W. Craig, J. Males, and K. Ono.  ``Distributions on partitions arising from Hilbert schemes and hook lengths''.  {\it Preprint} (2021).  \texttt{https://arxiv.org/abs/2109.10394}
	
	\bibitem{Bringmannbook}  K. Bringmann, A. Folsom, K. Ono, and L. Rolen.  {\it Harmonic Maass forms and mock modular forms}.  American Mathematical Society Colloquium Publications.  AMS, 2017.
	
	\bibitem{BMJS}  K. Bringmann, C. Jennings-Shaffer, and K. Mahlburg. ``On a Tauberian theorem of Ingham and Euler--Maclaurin summation''.  {\it Ramanujan J.} (2021)
	
	\bibitem{CaTir}  D. Cakmak and A. Tiryaki.  ``Mean value theorem for holomorphic functions''.  {\it Electron. J. Equ.} (2012).
	
	\bibitem{Craig}  W. Craig.  ``On the number of parts in congruence classes for partitions into distinct parts''.  {\it preprint} (2021) \texttt{https://arxiv.org/abs/2110.15835}
	
	\bibitem{hardyramanujan}  G. H. Hardy and S. Ramanujan.  ``Asymptotic formulae in combinatory analysis.''  {\it Proc. London Math. Soc.} {\bf 17} (1918), pp. 75--115.
	
	\bibitem{HardyWright}  G. H. Hardy and E. M. Wright.  {\it Introduction to the Theory of Numbers, fourth edition.}  Oxford University Press, 1960.
	
	\bibitem{Ingham}  A. E. Ingham.  ``A Tauberian theorem for partitions''.  {\it Annals of Mathematics} (1941).
	
	\bibitem{IwanKow}  H. Iwaniec and E. Kowalski.  {\it Analytic Number Theory.}  Oxford University Press, 2004.
	
	\bibitem{Males}  J. Males.  ``Asymptotic equidistribution and convexity for partition ranks''.  {\it Ramanujan J.}  {\bf 54} (2021), pp. 397--413.
	
	\bibitem{Parry} D. Parry.  ``A polynomial variation on Meinardus' theorem''.  {\it Int. J. of Number Theory}
	{\bf 11} (2015), pp. 251–-268.
	
	\bibitem{Pinsky}  M. Pinsky.  {\it Introduction to Fourier Analysis and Wavelets.}  Graduate Studies in Mathematics.  AMS, 2009.
	
	\bibitem{Rademacher}  H. Radamacher.  ``On the partition function $p(n)$''.  {\it Proc. London Math. Soc.}  {\bf 43} (1937), pp. 241--254.
	
	\bibitem{Stein} E. Stein and R. Shakarchi.  {\it Fourier Analysis: An Introduction.} Princeton University Press,  2003.
	
	\bibitem{Asymptbook}  N. Temme.  {\it Asymptotic Methods for Integrals.}  Vol. {\bf 6} Series in Analysis.  World Scientific, 2015.
	
	\bibitem{Tenenbaum}  G. Tenenbaum.  {\it Introduction to Analytic and Probabilistic Number Theory, third edition.}  Vol. {\bf 163}.  Graduate Studies in Mathematics.  AMS, 2015.
	
	\bibitem{Wright}  E. M. Wright.  ``Asymptotic partition formulae III.  Partitions into $k$-th powers''.  {\it Acta Math.} {\bf 63} (1934), pp. 143--191.
\end{thebibliography}
\end{document}